\documentclass[11pt]{amsart}

\usepackage{graphics}
\usepackage{color}
\usepackage[a4paper, margin=3cm]{geometry}

\usepackage{amssymb,enumerate}
\usepackage{amsmath}
\usepackage{bbm}           
\usepackage{bm}
\usepackage{eso-pic,graphicx}
\usepackage{tikz}
\usepackage{cite}
\usepackage{esint}
\usepackage[colorlinks=true, pdfstartview=FitV, linkcolor=blue, citecolor=blue, urlcolor=blue]{hyperref}

\usepackage{graphicx}

\usepackage{amsmath,amsfonts}
\usepackage{rotating}
\usepackage{amssymb}
\usepackage{verbatim}
\usepackage{rotating}
\usepackage{mathrsfs}
\DeclareSymbolFont{largesymbols}{OMX}{cmex}{m}{n}
\makeatletter
\def\Ddots{\mathinner{\mkern1mu\raise\p@
\vbox{\kern7\p@\hbox{.}}\mkern2mu
\raise4\p@\hbox{.}\mkern2mu\raise7\p@\hbox{.}\mkern1mu}}
\makeatother
\newtheorem*{thA}{Theorem A}

\newtheorem*{thC}{Theorem C}

\def\XXint#1#2#3{{\setbox0=\hbox{$#1{#2#3}{\int}$}
\vcenter{\hbox{$#2#3$}}\kern-.5\wd0}}

\begin{document}

\newtheorem{definition}{Definition}
\newtheorem{theorem}[definition]{Theorem}
\newtheorem{proposition}[definition]{Proposition}
\newtheorem{conjecture}[definition]{Conjecture}
\def\theconjecture{\unskip}
\newtheorem{corollary}[definition]{Corollary}
\newtheorem{lemma}[definition]{Lemma}
\newtheorem{claim}[definition]{Claim}
\newtheorem{sublemma}[definition]{Sublemma}
\newtheorem{observation}[definition]{Observation}
\theoremstyle{definition}

\newtheorem{notation}[definition]{Notation}
\newtheorem{remark}[definition]{Remark}
\newtheorem{question}[definition]{Question}

\newtheorem{example}[definition]{Example}
\newtheorem{problem}[definition]{Problem}
\newtheorem{exercise}[definition]{Exercise}
 \newtheorem{thm}{Theorem}
 \newtheorem{cor}[thm]{Corollary}
 \newtheorem{lem}{Lemma}[section]
 \newtheorem{prop}[thm]{Proposition}
 \theoremstyle{definition}
 \newtheorem{dfn}[thm]{Definition}
 \theoremstyle{remark}
 \newtheorem{rem}{Remark}
 \newtheorem{ex}{Example}
 \numberwithin{equation}{section}
\def\C{\mathbb{C}}
\def\R{\mathbb{R}}
\def\Rn{{\mathbb{R}^n}}
\def\Rns{{\mathbb{R}^{n+1}}}
\def\Sn{{{S}^{n-1}}}
\def\sn{{\mathbb{S}^{n-1}}}
\def\M{\mathbb{M}}
\def\N{\mathbb{N}}
\def\Q{{\mathbb{Q}}}
\def\Z{\mathbb{Z}}
\def\X{\mathbb{X}}
\def\Y{\mathbb{Y}}
\def\F{\mathcal{F}}
\def\L{\mathcal{L}}
\def\S{\mathcal{S}}
\def\supp{\operatorname{supp}}
\def\essi{\operatornamewithlimits{ess\,inf}}
\def\esss{\operatornamewithlimits{ess\,sup}}

\numberwithin{equation}{section}
\numberwithin{thm}{section}
\numberwithin{definition}{section}
\numberwithin{equation}{section}

\def\earrow{{\mathbf e}}
\def\rarrow{{\mathbf r}}
\def\uarrow{{\mathbf u}}
\def\varrow{{\mathbf V}}
\def\tpar{T_{\rm par}}
\def\apar{A_{\rm par}}

\def\reals{{\mathbb R}}
\def\torus{{\mathbb T}}
\def\t{{\mathcal T}}
\def\heis{{\mathbb H}}
\def\integers{{\mathbb Z}}
\def\z{{\mathbb Z}}
\def\naturals{{\mathbb N}}
\def\complex{{\mathbb C}\/}
\def\distance{\operatorname{distance}\,}
\def\support{\operatorname{support}\,}
\def\dist{\operatorname{dist}\,}
\def\Span{\operatorname{span}\,}
\def\degree{\operatorname{degree}\,}
\def\kernel{\operatorname{kernel}\,}
\def\dim{\operatorname{dim}\,}
\def\codim{\operatorname{codim}}
\def\trace{\operatorname{trace\,}}
\def\Span{\operatorname{span}\,}
\def\dimension{\operatorname{dimension}\,}
\def\codimension{\operatorname{codimension}\,}
\def\nullspace{\scriptk}
\def\kernel{\operatorname{Ker}}
\def\ZZ{ {\mathbb Z} }
\def\p{\partial}
\def\rp{{ ^{-1} }}
\def\Re{\operatorname{Re\,} }
\def\Im{\operatorname{Im\,} }
\def\ov{\overline}
\def\eps{\varepsilon}
\def\lt{L^2}
\def\diver{\operatorname{div}}
\def\curl{\operatorname{curl}}
\def\etta{\eta}
\newcommand{\norm}[1]{ \|  #1 \|}
\def\expect{\mathbb E}
\def\bull{$\bullet$\ }

\def\blue{\color{blue}}
\def\red{\color{red}}

\def\xone{x_1}
\def\xtwo{x_2}
\def\xq{x_2+x_1^2}
\newcommand{\abr}[1]{ \langle  #1 \rangle}

\newcommand{\Norm}[1]{ \left\|  #1 \right\| }
\newcommand{\set}[1]{ \left\{ #1 \right\} }
\newcommand{\ifou}{\raisebox{-1ex}{$\check{}$}}
\def\one{\mathbf 1}
\def\whole{\mathbf V}
\newcommand{\modulo}[2]{[#1]_{#2}}
\def \essinf{\mathop{\rm essinf}}
\def\scriptf{{\mathcal F}}
\def\scriptg{{\mathcal G}}
\def\m{{\mathcal M}}
\def\scriptb{{\mathcal B}}
\def\scriptc{{\mathcal C}}
\def\scriptt{{\mathcal T}}
\def\scripti{{\mathcal I}}
\def\scripte{{\mathcal E}}
\def\V{{\mathcal V}}
\def\scriptw{{\mathcal W}}
\def\scriptu{{\mathcal U}}
\def\scriptS{{\mathcal S}}
\def\scripta{{\mathcal A}}
\def\scriptr{{\mathcal R}}
\def\scripto{{\mathcal O}}
\def\scripth{{\mathcal H}}
\def\scriptd{{\mathcal D}}
\def\scriptl{{\mathcal L}}
\def\scriptn{{\mathcal N}}
\def\scriptp{{\mathcal P}}
\def\scriptk{{\mathcal K}}
\def\frakv{{\mathfrak V}}
\def\v{{\mathcal V}}
\def\C{\mathbb{C}}
\def\D{\mathcal{D}}
\def\R{\mathbb{R}}
\def\Rn{{\mathbb{R}^n}}
\def\rn{{\mathbb{R}^n}}
\def\Rm{{\mathbb{R}^{2n}}}
\def\r2n{{\mathbb{R}^{2n}}}
\def\Sn{{{S}^{n-1}}}
\def\bbM{\mathbb{M}}
\def\N{\mathbb{N}}
\def\Q{{\mathcal{Q}}}
\def\Z{\mathbb{Z}}
\def\F{\mathcal{F}}
\def\L{\mathcal{L}}
\def\G{\mathscr{G}}
\def\ch{\operatorname{ch}}
\def\supp{\operatorname{supp}}
\def\dist{\operatorname{dist}}
\def\essi{\operatornamewithlimits{ess\,inf}}
\def\esss{\operatornamewithlimits{ess\,sup}}
\def\dis{\displaystyle}
\def\dsum{\displaystyle\sum}
\def\dint{\displaystyle\int}
\def\dfrac{\displaystyle\frac}
\def\dsup{\displaystyle\sup}
\def\dlim{\displaystyle\lim}
\def\bom{\Omega}
\def\om{\omega}

\author[J. Tan]{Jiawei Tan}
\address{Jiawei Tan:
School of Mathematical Sciences \\
Beijing Normal University \\
Laboratory of Mathematics and Complex Systems \\
Ministry of Education \\
Beijing 100875 \\
People's Republic of China}
\email{jwtan@mail.bnu.edu.cn}

\author[Q. Xue]{Qingying Xue$^{*}$}
\address{Qingying Xue:
	School of Mathematical Sciences \\
	Beijing Normal University \\
	Laboratory of Mathematics and Complex Systems \\
	Ministry of Education \\
	Beijing 100875 \\
	People's Republic of China}
\email{qyxue@bnu.edu.cn}


\thanks{The authors were partly supported by the National Key R\&D Program of China (No. 2020YFA0712900) and NNSF of China (No. 12271041).
\thanks{$^{*}$ Corresponding author, e-mail address: qyxue@bnu.edu.cn}}

\date{\today}
\title[ COMPOSITION OF ROUGH SINGULAR INTEGRAL OPERATORS  ]
{\bf Composition of rough singular integral operators on rearrangement invariant Banach type spaces}

\begin{abstract}
Let $\Omega$ be a homogeneous function of degree zero and enjoy the vanishing condition on the unit sphere $\mathbb{S}^{n-1}(n\geq 2)$. Let $T_{\Omega}$ be the convolution singular integral operator with kernel ${\Omega(x)}{|x|^{-n}}$. In this paper, when $\Omega \in L^{\infty}(\mathbb {S}^{n-1})$,  we consider the quantitative weighted bounds of the composite operators of $T_{\Omega}$ on rearrangement invariant Banach function spaces. These spaces contain the classical Lorentz spaces and Orlicz spaces as special examples.  Weighted boundedness of the composite operators on rearrangement invariant quasi-Banach spaces were also given.
\end{abstract}\maketitle

\section{Introduction and main results}
This paper aims to establish the quantitative weighted boundedness for the composition of rough singular integral operators in rearrangement invariant Banach spaces and quasi-Banach spaces.
It is worthy to pointing out that the classical Lorentz spaces, Orlicz spaces and  the Marcinkiewicz spaces are special examples of these spaces. The study of these rearrangement invariant function spaces has a long history. Indeed, in 1955, Lorentz \cite{lor} first showed that the Hardy-Littlewood maximal operator $M$ is bounded on rearrangement invariant Banach function space $\mathbb{X}$ if and only if $p_{\mathbb{X}}>1$. Subsequently, Boyd \cite{boy} proved that the Hilbert transform $H$ is also bounded on $\mathbb{X}$ if and only if $1<p_{\mathbb{X}} \leq q_{\mathbb{X}}<\infty .$  Here $p_{\mathbb{X}}$ and $q_{\mathbb{X}}$ denote the Boyd indices of $\mathbb{X}$ (see Section 2.1 below). These results were originally proved for Banach spaces, but they were later generalized to the quasi-Banach case with the same restrictions on Boyd's index  in \cite{mon}. Since then many other contributions came to enrich the literature on this subject, we refer the readers to \cite{ben, edm, cur} and the references therein. In particular, by using sparse domination, Anderson and Hu \cite{and} obtained the quantitative weighted bounds for the maximal truncated  singular integral operator on rearrangement invariant Banach spaces.

We now give a brief review  on  the study of rough singular integrals. Let $\Omega$ be a homogeneous function of degree zero, $ \Omega \in L^{1}(\mathbb{S}^{n-1})
$
and satisfy the vanishing condition on the unit sphere $\mathbb{S}^{n-1}(n\geq 2)$ as follows
\begin{equation}\label{e.1}
\int_{\mathbb{S}^{n-1}} \Omega(y) d \sigma(y)=0,
\end{equation}
where $d \sigma(y)$ denotes the Lebesgue measure with restrictions on $\mathbb {S}^{n-1}$.
 In 1956, Calder\'{o}n and Zygmund  \cite{cal2} introduced  the following rough homogeneous singular integral operator
\begin{equation}
T_{\Omega} f(x)=\text {p.v.} \int_{\mathbb{R}^{n}} \frac{\Omega\left(y/|y|\right)}{|y|^{n}} f(x-y) d y.
\end{equation}
Using the method of rotation, Calder\'{o}n and Zygmund \cite{cal2} demonstrated the $L^p$ $(1<p<\infty)$ boundedness of $T_{\Omega}$
whenever $\Omega $ is odd and
$\Omega \in L^1(\mathbb{S}^{n-1})$, or $\Omega$ is even, $\Omega\in L\log L(\mathbb{S}^{n-1})$ and satisfies (\ref {e.1}).
In 1979, Connett \cite{C1979}, Ricci and Weiss \cite{ric} independently showed that \(\Omega\in H^1(\sn)\) is sufficient to warrant the $L^p$ boundedness of $T_{\Omega}$. Here \(H^1(\sn)\) denotes the Hardy space on \(\sn\) which contains \(L\log L(\sn)\) as a proper subspace.

Conside the weak endpoint case and the weighted case. This area has flourished and has been enriched by  several important works.
Among them are the celebrated works of Christ,  Christ and Rubio~de Francia, Hofmann, Seeger, and Tao for weak type $(1,1)$ bounds of $T_{\Omega}$.  In particular,  in 1996, Seeger \cite{see} proved that $T_{\Omega}$ is bounded from $L^{1}(\mathbb{R}^{n})$ to $L^{1, \infty}(\mathbb{R}^{n})$ with a sufficient condition $\Omega \in L \log L\left(\sn\right).$
For the weighted cases, it was Duoandikoetxea and Rubio de Francia \cite{duo} who obtained the $L^{p}(\mathbb{R}, w(x) d x)$-boundedness of $T_{\Omega}$ for $1<p<\infty$ provided that $\Omega \in L^{\infty}(\mathbb{S}^{n-1})$ and $w $ is a Muckenhoupt $A_{p}$ weight. This result was later improved in \cite{duo1} and \cite{wat}. It is worth mentioning that, in 2017, Hyt\"{o}nen et al. \cite{hyt1} obtained the quantitative weighted boundedness of $T_{\Omega}.$
For other works related to $T_{\Omega}$, we refer the readers to see \cite{chr, ric, lik,GS1999} and the references therein.

In general,  there are two distinct approaches in the study of singular integral operators. Consider it as a principal value operator of convolution type or as an operator of Fourier multipliers defined by
$$
\widehat{T_{m} f}(\xi)=m(\xi) \hat{f}(\xi),
$$
where $m \in L^{\infty}\left(\mathbb{R}^{n}\right)$ and $\hat{f}$ denotes the Fourier transform of $f$.
Let $\Omega \in L\log L(\mathbb{S}^{n-1})$ be homogeneous of degree zero and satisfy the vanishing condition (\ref{e.1}). Then the following identity transformation relationship holds between $m$ and $\Omega$
$$
m(\xi)=\int_{\mathbb{S}^{n-1}} \Omega\left(y^{\prime}\right)\left(\log \frac{1}{\left|\xi \cdot y^{\prime}\right|}-\frac{i \pi}{2} \operatorname{sgn}\left(\xi \cdot y^{\prime}\right)\right) d \sigma(y).
$$
 However, unfortunately, this identity does not provide an exact correspondence between various auxiliary conditions assumed on $m$ and $\Omega$.\par
One of the fundamental questions in operator theory is what the composition of two operators is. This question is of great importance and attracts lots of attention. Indeed, it was known that the composition of singular integral operators arises typically in the study of algebra of singular integral (see \cite{cal0, cal3}) and the non-coercive boundary-value problems for elliptic equations (see \cite{nag,pho}).
The answer to this question for $T_{\Omega}$ is easily obtained by using the Fourier multiplier representation, which states that
the composition of two singular integral operators is an operator of the same form and  the multiplier of the composition is the product of the two multipliers (see \cite{str}). This answer is so elegant and useful that it actually forms the basis for the calculus of pseudo-differential operators. {\textbf {What if it is in the form of a principal value integral?}}   Coifman and Meyer \cite{coi} considered the composition of classical Calder\'{o}n-Zygmund operators and pointed out that
if $T_{1}, T_{2}$ are two Calder\'{o}n-Zygmund operators, $T_{2}^{*}$ be the adjoint operator of $T_{2}$ and $T_{1}(1)=T_{2}^{*}(1)=0$, then the composite operator $T_{1} T_{2}$ is also a Calder\'{o}n-Zygmund operator. It then follows that the composition of Calder\'{o}n-Zygmund operators is still strong $(p,p)$ type and weak $(1,1)$ type. In 2018, Benea and Bernicot \cite{bene} used the sparse domination method to reduce the above conditions to $T_1(1) = 0$ or $T^*_2 (1) = 0,$ and the weighted boundedness of the composite operator for the Calder\'{o}n-Zygmund operators can also be obtained. In addition, previous result for Hardy-Littlewood maximal operators $M$ were obtained by Carozza and Passarelli di Napoli \cite{car}. They showed that the following weak type endpoint estimates hold for the composition of  $M$,
$$
\left|\left\{x \in \mathbb{R}^{n}: M^{k} f(x)>\lambda\right\}\right| \lesssim \int_{\mathbb{R}^{n}} \Psi_{k-1}\left(\frac{|f(x)|}{\lambda}\right) d x, \quad 0\leq \beta <\infty,
$$
where $M^k$ is the $k$-th iterations of $M$ and  $\Psi_{\beta}(t)=t \log ^{\beta}(\mathrm{e}+t).$

By sparse domination, the results in \cite{bene} imply the conclusion that if $T_{1}, T_{2}$ be two Calder\'{o}n-Zygmund operators with $T_{1}(1)=0$, then for any $1<p< \infty, 1<q< p$, and $w \in A_{p / q}\left(\mathbb{R}^{n}\right)$,
$$
\left\|T_{1} T_{2} f\right\|_{L^{p}\left(\mathbb{R}^{n}, w\right)} \lesssim[w]_{A_{p / q}}^{\max \left\{\frac{1}{p-q}, 1\right\}}\|f\|_{L^{p}\left(\mathbb{R}^{n}, w\right)},
$$
where the precise definitions of $A_{p}\left(\mathbb{R}^{n}\right)$ weight and $A_{p}$ constants are listed in Section \ref {Sect 2}.
It was Hu \cite{hu3} who proved the weighted bound for the composite operator $T_1T_2$ without the
assumption $T_1(1) = 0.$ Recently, Hu \cite{hu1} established the weighted weak type endpoint estimate for $T_1T_2.$ Still more recently, using the method of sparse domination, more accurate weighted estimate for the composition of rough homogeneous singular integral operators were obtained by Hu, Lai and Xue \cite{hu2}.
\begin{thA}[\cite{hu2}]
 Let $\Omega_{1}, \Omega_{2}$ be homogeneous of degree zero, have mean value zero and $\Omega_{1}, \Omega_{2} \in L^{\infty}\left(\mathbb{S}^{n-1}\right)$. Then for $p \in(1, \infty)$ and $w \in A_{p}\left(\mathbb{R}^{n}\right)$,
$$
\begin{aligned}
\left\|T_{\Omega_{1}} T_{\Omega_{2}} f\right\|_{L^{p}\left(\mathbb{R}^{n}, w\right)} \lesssim & {[w]_{A_{p}}^{\frac{1}{p}}\left([w]_{A_{\infty}}^{\frac{1}{p^{\prime}}}+[\sigma]_{A_{\infty}}^{\frac{1}{p}}\right)\left(
[\sigma]_{A_{\infty}}+[w]_{A_{\infty}}\right) } \\
& \times \min \left\{[\sigma]_{A_{\infty}},[w]_{A_{\infty}}\right\}\|f\|_{L^{p}\left(\mathbb{R}^{n}, w\right)},
\end{aligned}
$$
where $p^{\prime}=p /(p-1), \sigma=w^{-1 /(p-1)}.$
\end{thA}
As we mentioned in the beginning of this paper, the main purpose of this paper is to obtain the boundedness of the composition for rough singular integral operators $T_{\Omega_{1}} T_{\Omega_{2}}$ on rearrangement invariant Banach spaces (RIBFS in the sequel) and quasi-Banach spaces (RIQBFS in the sequel) whenever both $\Omega_{1}$ and $\Omega_{2}$ belong to $L^{\infty}\left(\mathbb{S}^{n-1}\right)$.

We summary our first  main result as follows.

\begin{theorem}\label{thm1.1}
Let $\Omega_{1}, \Omega_{2}$ be homogeneous of degree zero, have the vanishing moment (\ref{e.1}) and $\Omega_{1}, \Omega_{2} \in L^{\infty}\left(\mathbb{S}^{n-1}\right)$. Let $1<r<\infty$ and $\X$ be a RIBFS with $1 <p_{\X}\leq q_{\X}< \infty$, then there exist $q,q_0 >1$ such that
\begin{equation*}\left\|T_{\Omega_1}T_{\Omega_2}f\right\|_{\X(w)}\lesssim\left\{\begin{array}{ll}
[w]_{A_\infty}^2  [w]_{A_{p_\X/r}}^{\frac{2}{rq}} \left\| f \right\|_{\X(w)}, &\text{ if } r<p_\X\leq q_\X, w\in A_{\frac{p_\X}{r}}; \\
{[w]_{A_\infty}^2}[w]_{A_{p_\X}}^{\frac{1}{p_\X}}\left([w]_{A_\infty}+[w]_{A_{p_\X}}^{\frac{1}{p_\X}}\right)
\left\|f\right\|_{\X(w)},& \text{ if } 1<p_\X<q_0, w\in A_{p_\X}.
\end{array}\right.
\end{equation*}
\end{theorem}
\begin{remark}
If we set $\X=L^p$ with $1<p<\infty$, then $p_{\X}= q_{\X}=p$ and the result in Theorem \ref{thm1.1} covers the conclusion in Theorem A as a special case. Furthermore, to the best knowledge of the author, even the quantitative weighted estimates for a single rough singular integral operator $T_\Omega$ on $\X$ is new.
\end{remark}
As a consequence of Theorem \ref{thm1.1}, it follows that:
\begin{corollary}\label{cor1.1}
Let $\Omega_{1}, \Omega_{2}$ be homogeneous of degree zero, have mean value zero and $\Omega_{1}, \Omega_{2} \in L^{\infty}\left(\mathbb{S}^{n-1}\right)$. Let $1<r<\infty$ and $\X$ be a RIQBFS, which is $p$-convex for some $p > 0$. If $p <p_{\X}\leq q_{\X}< \infty$, then there exist $q,q_0 >1$ such that
\begin{equation*}\left\||T_{\Omega_1}T_{\Omega_2}f|^{\frac{1}{p}}\right\|_{\X(w)}\lesssim\left\{\begin{array}{ll}
[w]_{A_\infty}^\frac{2}{p}  [w]_{A_{\frac{p_\X}{pr}}}^{\frac{2}{prq}} \left\| |f|^{\frac{1}{p}} \right\|_{\X(w)}, & pr<p_\X, w\in A_{\frac{p_\X}{pr}}; \\
{[w]_{A_\infty}^\frac{2}{p}}[w]_{A_{\frac{p_\X}{p}}}^{\frac{1}{p_\X}}\left([w]^{\frac{1}{p}}_{A_\infty}+[w]_{A_{\frac{p_\X}{p}}}^
{\frac{1}{p_\X}}\right)
\left\||f|^{\frac{1}{p}}\right\|_{\X(w)},& p_\X<pq_0, w\in A_{\frac{p_\X}{p}}.
\end{array}\right.
\end{equation*}
\end{corollary}
It is well known that some important facts may be significantly different between Banach spaces and quasi-Banach spaces. For example, the direction of the H\"{o}lder's inequality in $L^p (0<p<1)$ and $L^q (q\geq 1)$  is opposite. When $\X$ is a space of rearrangement invariant quasi-Banach type, we obtain the following quantitative weighted bounds for the composition of rough singular integral operators.
\begin{theorem}\label{thm1.2}
Let $\Omega_{1}, \Omega_{2}$ be homogeneous of degree zero, have mean value zero and $\Omega_{1}, \Omega_{2} \in L^{\infty}\left(\mathbb{S}^{n-1}\right)$.  Let $\X$ be a RIQBFS, which is p-convex with $0<p \leq 1$. If $ 1<p_{\X}\leq q_{\X}< 2p-\frac{1}{1+p_{\X}} p$, then for every $w \in A_{{p_{\X}}}$,
\begin{equation*}
\left\|T_{\Omega_1}T_{\Omega_2}f\right\|_{\X(w)}\lesssim \left([w]_{A_\infty}^{1+\frac{1}{p}}+[w]_{A_\infty}^{2+\frac{1}{p}}\right)  \left([w]_{A_{p_\X}}^{\frac{1}{p_\X}}+[w]_{A_{p_\X}}^{\frac{2}{p_\X }}\right)\left\|f \right\|_{\X(w)}.
\end{equation*}
\end{theorem}
\vspace{0.2cm}

This paper is organized as follows. In Section \ref  {Sect 2}, we present some lemmas and related definitions, such as RIBFS and RIQBFS, dyadic cubes, some maximal operators and bi-sublinear sparse operators. The proofs of Theorem \ref{thm1.1} and Corollary \ref{cor1.1} will be given in Section \ref {Sect 3}. In Section \ref {Sect 4}, we demonstrate Theorem \ref{thm1.2}.  An application of Theorem \ref{thm1.1} will be given in Section \ref {Sect 5}.
\par

In what follows, $C$ always denotes a positive constant that is independent of the main parameters involved but whose value may differ from line to line. For any $a, b \in \mathbb{R}, a \lesssim b$ $(a \gtrsim b,$ respectively) denotes that there exists a constant $C>0$ such that $a \leq C b;$ and $a \simeq b$ denotes $a \lesssim b$ and $b \lesssim a.$ $p^{\prime}$ will always denote the conjugate of $p,$ namely, $1 / p+1 / p^{\prime}=1$.

\section{Preliminary}\label {Sect 2}
First, we recall some basic properties for RIBFS, RIQBFS, sparse family and Orlicz maximal operators.
\subsection{ RIBFS and RIQBFS}
Let's start with some simple definitions. \par
$\circ$ \textbf{Basic definitions of RIBFS.} Let $\mathcal{M}$ be the set of measurable functions on $\left(\mathbb{R}^{n}, d x\right)$ and $\mathcal{M}^{+}$ be the nonnegative ones. A rearrangement invariant Banach norm is a mapping $\rho: \mathcal{M}^{+} \mapsto[0, \infty]$ such that the following properties hold:
\medskip
\begin{enumerate}[i).]
	\item $\rho(f)=0 \Leftrightarrow f=0,$ a.e.; $\rho(f+g) \leq \rho(f)+\rho(g) ; \rho(a f)=a \rho(f)$ for $a \geq 0$;
	\item If $0 \leq f \leq g,$ a.e., then $\rho(f) \leq \rho(g)$;
	\item If $f_{n} \uparrow f,$ a.e., then $\rho\left(f_{n}\right) \uparrow \rho(f)$;
	\item If $E$ is a measurable set such that $|E|<\infty,$ then $\rho\left(\chi_{E}\right)<\infty,$ and $\int_{E} f d x \leq$ $C_{E} \rho(f),$ for some constant $0<C_{E}<\infty,$ depending on $E$ and $\rho,$ but independent of $f$;
	\item $\rho(f)=\rho(g)$ if $f$ and $g$ are equimeasurable, that is, $d_{f}(\lambda)=d_{g}(\lambda), \lambda \geq 0$ where $d_{f}\left(d_{g}\right.$ respectively) denotes the distribution function of $f$ ($g$ respectively).
\end{enumerate}
\medskip
By means of $\rho,$ a rearrangement invariant Banach function space (RIBFS) $\X$ can be defined:
$$\mathbb{X}=\left\{f \in \mathcal{M}:\|f\|_{\mathbb{X}}:=\rho(|f|)<\infty\right\}.$$
Let $\mathbb{X}^{\prime}$ be the associated space of $\mathbb{X}$, which is also a Banach function space given by
$$
\mathbb{X}^{\prime}=\left\{f \in \mathcal{M}:\|f\|_{\mathbb{X}^{\prime}}:=\sup \left\{\int_{\mathbb{R}^{n}} f g d x: g \in \mathcal{M}^{+}, \rho(g) \leq 1\right\}<\infty\right\}.
$$
Note that in the present setting, $\mathbb{X}$ is a RIBFS if and only if $\mathbb{X}^{\prime}$ is a $\mathrm{RIBFS}$ (\cite[Chapter 2, Corollary 4.4]{ben}). By definition, the following generalized H\"{o}lder's inequality holds for any $f \in \mathbb{X}, g \in \mathbb{X}^{\prime}$ :
$$
\int_{\mathbb{R}^{n}}|f g| dx \leq\|f\|_{\mathbb{X}}\|g\|_{\mathbb{X}^{\prime}}.
$$
A key fact in a RIBFS $\X$ is that the Lorentz-Luxemburg theorem holds:
$$\|f\|_{\mathbb{X}}=\sup \left\{\left|\int_{\mathbb{R}^{n}} f g d x\right|: g \in \mathbb{X}^{\prime},\|g\|_{\mathbb{X}^{\prime}} \leq 1\right\}.$$\par
Recall that the decreasing rearrangement function $f^{*}$ is defined by
$$
f^{*}(t)=\inf \left\{\lambda \geq 0: d_{f}(\lambda) \leq t\right\}, t \geq 0.
$$
An important property of $f^{*}$ is that it is equimeasurable with $f$. This allows one to obtain a representation of $\X$, i.e., Luxemburg’s representation theorem (\cite[Chapter 2, Theorem 4.10]{ben}), which asserts that there exists a RIBFS $\overline{\mathbb{X}}$ over $\left(\mathbb{R}^{+}, d t\right)$ such that $f \in \mathbb{X}$ if and only if $f^{*} \in \overline{\mathbb{X}}$, and in this case $\|f\|_{\mathbb{X}}=\left\|f^{*}\right\|_{\overline{\mathbb{X}}}$. From this it can be seen that the mapping $f \mapsto f^{*}$ is an isometry.
In addition, notice that $\overline{\mathbb{X}}^{\prime}=\overline{\mathbb{X}^{\prime}}$ and $\|f\|_{\mathbb{X}^{\prime}}=\left\|f^{*}\right\|_{\overline{\mathbb{X}}^{\prime}}$ hold for the associated space.\par

\medskip
$\circ$ \textbf{Weighted versions of RIBFS $\X$.}
Before we consider weighted versions of RIBFS $\X$, we need to recall the definitions of Muckenhoupt weights. Let $w$ be a non-negative locally integrable function defined on $\mathbb{R}^{n}.$ We say that a weight $w$ belongs to the Muckenhoupt class $A_{p}$ with $1<p<\infty,$ if
$$
[w]_{A_{p}}:=\sup _{Q} \left(\frac{1}{|Q|} \int_{Q} w(x) d x\right)\left(\frac{1}{|Q|} \int_{Q} w(x)^{1-p^{\prime}} d x\right)^{p-1} <\infty,
$$
and for $w \in A_{\infty},$
$$
[w]_{A_{\infty}}:=\sup _{Q} \frac{1}{w(Q)} \int_{Q} M\left(w \chi_{Q}\right)(x) d x,
$$
where the supremum is taken over all cubes $Q \subset \mathbb{R}^{n}.$
The $A_{1}$ constant is defined by
$$
[w]_{A_{1}}:=\sup _{x \in \mathbb{R}^{n}} \frac{M w(x)}{w(x)},
$$
where $M$ is the Hardy-Littlewood maximal operator.\par

Some important properties of weights are listed in the following lemmas.

\begin{lemma}[\cite{per}]\label{lem1.4}
	Let $1<p<\infty$ and let $w \in A_{p}.$ Then $w \in A_{p-\varepsilon}$ with
	$$
	\varepsilon:=\varepsilon (p)=\frac{p-1}{1+2^{n+1}[\sigma]_{A_{\infty}}}
	$$
	where $\sigma=w^{1-p^{\prime}}$ is the dual weight. Furthermore
	$
	[w]_{A_{p-\varepsilon}} \leq 2^{p-1}[w]_{A_{p}}.
	$
\end{lemma}

\begin{lemma}[\cite{hyt}]\label{lem1.1}
 Let $w \in A_{\infty}\left(\mathbb{R}^{n}\right)$. Then for any cube $Q$ and $\delta \in\left(1,1+\frac{1}{2^{11+n}[w]_{A_{\infty}}}\right]$,
$$
\left(\frac{1}{|Q|} \int_{Q} w^{\delta}(x) \mathrm{d} x\right)^{\frac{1}{\delta}} \leq \frac{2}{|Q|} \int_{Q} w(x) \mathrm{d} x.
$$
\end{lemma}
\par
We now give the weighted version of the RIBFS $\X$. First, the distribution function and the decreasing rearrangement with respect to $w$ are defined by
$$w_{f}(\lambda)=w\left(\{x \in \mathbb{R}^{n}:|f(x)|>\lambda\}\right) ; \quad f_{w}^{*}(t)=\inf \left\{\lambda \geq 0: w_{f}(\lambda) \leq t\right\}.$$
In this way, the weighted version of the space $\X$ is given by
$$\mathbb{X}(w)=\left\{f \in \mathcal{M}:\|f\|_{\mathbb{X}(w)}:=\left\|f_{w}^{*}\right\|_{\overline{\mathbb{X}}}<\infty\right\}.$$
Then, the same procedure applying on the associate spaces yields $\mathbb{X}^{\prime}(w)=\mathbb{X}(w)^{\prime}$ (see \cite[p. 168]{cur}).\par

$\circ$ \textbf{ Boyd indices and $r$ exponent.} Next, we recall the Boyd indices of a RIBFS, which are closely related to some interpolation properties, see \cite[Chapter 3]{ben} for more details. We start with the dilation operator $D_{t}$ of $\X$,
$$ D_{t} f(s)=f\left(\frac{s}{t}\right), 0<t<\infty, f \in \overline{\mathbb{X}} ,$$
and its norm
$$ h_{\mathbb{X}}(t)=\left\|D_{t}\right\|_{\overline{\mathbb{X}} \mapsto \overline{\mathbb{X}}},  0<t<\infty .$$
The lower and upper Boyd indices are defined, respectively, by the following form:
$$p_{\mathbb{X}}=\lim _{t \rightarrow \infty} \frac{\log t}{\log h_{\X}(t)}=\sup _{1<t<\infty} \frac{\log t}{\log h_{\X}(t)}, \quad q_{\mathbb{X}}=\lim _{t \rightarrow 0^{+}} \frac{\log t}{\log h_{\mathbb{X}}(t)}=\inf _{0<t<1} \frac{\log t}{\log h_{\mathbb{X}}(t)}.$$
A simple calculation shows that $1 \leq p_{\mathbb{X}} \leq q_{\mathbb{X}} \leq \infty,$ which follows from the fact that $h_{\mathbb{X}}(t)$ is submultiplicative. In order to give the above definition a general explanation, we consider a special case. If $\X =L^p$ with $1<p<\infty$, then $h_{\X}(t)=t^{\frac{1}{p}}$ and thus $p_{\X}=q_{\mathbb{X}}=p.$ \par
The relationship between the Boyd indices of $\mathbb{X}$ and $\mathbb{X}^{\prime}$ are as follows : $p_{\mathbb{X}^{\prime}}=\left(q_{\mathbb{X}}\right)^{\prime}$ and $q_{\mathbb{X}^{\prime}}=\left(p_{\mathbb{X}}\right)^{\prime},$ where $p$ and
$p^{\prime}$ are conjugate exponents.(see \cite[Chapter 11, Corollary 11.6]{mal}).\par

Now, we consider the following $r$ exponent of RIBFS $\mathbb{X}$ with $0<r<\infty$:
$$\mathbb{X}^{r}=\left\{f \in \mathcal{M}:|f|^{r} \in \mathbb{X}\right\},$$
and the norm $\|f\|_{\X ^{r}}=\left\||f|^{r}\right\|_{\X}^{\frac{1}{r}}$. By the definition of Boyd indices it is easy to verify that $p_{\mathbb{X}^{r}}=p_{\mathbb{X}} \cdot r$ and $q_{\mathbb{X}^{r}}=q_{\mathbb{X}} \cdot r$. \par

$\circ$ \textbf{The case of RIQBFS.}
Analogy to $L^p$ space, for each $r \geq 1,$ $\mathbb{X}^{r}$ is still a RIBFS when $\X$ is a RIBFS. However, if $0<r<1,$ the space $\mathbb{X}^{r}$ is not necessarily a Banach space (see \cite[p. 269]{cur}). Hence, it is natural to consider the quasi-Banach case.\par
To see this, we first give the definition of the quasi-Banach function norm. We say a mapping $\rho^{\prime}: \mathcal{M}^{+} \mapsto[0, \infty)$ is a rearrangement invariant quasi-Banach function norm if $\rho^{\prime}$ satisfies the basic condition \romannumeral1),\romannumeral2),\romannumeral3) and \romannumeral5) with the triangle inequality replaced by the quasi-triangle inequality as follows:
$$\rho^{\prime}(f+g) \leq C\left(\rho^{\prime}(f)+\rho^{\prime}(g)\right),$$
where $C$ is an absolute positive constant. Similarly, a rearrangement invariant quasi-Banach function space is a collection that consist of all measurable functions which satisfies $\rho^{\prime}(|f|)<\infty$. In order to get a better study in transformation between RIBFS and RIQBFS, we need to consider the following $p$-convex property with $p>0$ on $\X$ which is a RIQBFS (see \cite[p. 3]{gra1}):
$$\left\|\left(\sum_{j=1}^{N}\left|f_{j}\right|^{p}\right)^{\frac{1}{p}}\right\| _{\mathbb{X}} \lesssim\left(\sum_{j=1}^{N}\left\|f_{j}\right\|_{\mathbb{X}}^{p}\right)^{\frac{1}{p}}.$$

A very important conclusion states that the $p$-convex property is equivalent to the fact that $\mathbb{X}^{\frac{1}{p}}$ is a RIBFS (see \cite[p. 269]{cur}). According to the above results and using Lorentz-Luxemburg’s theorem again, we have
$$\|f\|_{\mathbb{X}} \simeq \sup \left\{\left(\int_{\mathbb{R}^{n}}|f(x)|^{p} g(x) d x\right)^{\frac{1}{p}}: g \in \mathcal{M}^{+},\|g\|_{\mathbb{Y}^{\prime}} \leq 1\right\},$$
where $\mathbb{Y}^{\prime}$ is the associated space of the RIBFS $\mathbb{Y}=\mathbb{X}^{\frac{1}{p}}$.  Simultaneously, for each
$w \in A_{\infty}$ and $0<r<\infty$, one may define $\mathbb{X}(w)$ for a RIQBFS $\X$ and it enjoys that $\mathbb{X}(w)^{r}=\mathbb{X}^{r}(w)$(see \cite[p. 269]{cur}).

\begin{remark}
	It is necessary for us to make a remark about the newly added $p$-convex property. As Grafakos and Kalton (\cite{gra1}) said ``all practical spaces are $p$-convex for some $p > 0$". This is because there are only very few spaces that do not satisfy the $p$-convex property for any $p>0$ (see \cite{joh}).
	
\end{remark}
\subsection{  Young function and Orlicz maximal operators} In this subsection, we present some fundamental facts about Young functions and Orlicz local averages which will play an important role in our analysis. We refer the readers to \cite{rao} for more information.\par
Firstly, let $\Phi$ be the set of functions $\phi:[0, \infty) \longrightarrow[0, \infty)$ which are non-negative, increasing and such that $\lim _{t \rightarrow \infty} \phi(t)=\infty $  and $\lim _{t \rightarrow 0} \phi(t)=0 .$ If $\phi \in \Phi$ is convex we say that $\phi$ is a Young function.
Next, we can define the average of the Luxemburg norm, namely $\phi$-norm, of $f$ over a cube $Q$ as
$$
\|f\|_{\phi(\mu), Q}:=\inf \left\{\lambda>0: \frac{1}{\mu(Q)} \int_{Q} \phi\left(\frac{|f(x)|}{\lambda}\right) d \mu \leq 1\right\}.
$$
For the sake of notation, if $\mu$ is the Lebesgue measure, we write $\|f\|_{\phi, Q}$,  and we denote $\|f\|_{\phi(w), Q},$ if $\mu=w d x$ is an absolutely continuous measure with respect to the Lebesgue measure.\par
Each Young function $\phi$ enjoys the following generalized H\"{o}lder's inequality:
$$\frac{1}{\mu(Q)} \int_{Q}|f g| d \mu \leq 2\|f\|_{\phi(\mu), Q}\|g\|_{\bar{\phi}(\mu), Q},$$
where $\bar{\phi}(t)=\sup _{s>0}\{s t-\phi(s)\}$ is the complementary function of $\phi$.\par

Then we can naturally represent the Orlicz maximal operator $M_{\phi} f$ associated to the Young function $\phi$ with the following form:
$$
M_{\phi} f(x):=\sup _{x \in Q}\|f\|_{\phi, Q}.
$$
Finally, we present some particular examples of maximal operators related to certain Young functions.
\begin{itemize}
	\item If $\phi(t)=t^{r}$ with $r>1,$ then $M_{\phi}=M_{r}$.
	\item If we consider $\phi(t)=t \log ^{\alpha}(e+t)$ with $\alpha>0 ,$ then $\bar{\phi}(t) \simeq e^{t^{1 / \alpha}}-1$ and we denote
	$M_{\phi}=M_{L (\log L)^{\alpha}}$. We have that $M \leq M_{\phi} \lesssim M_{r}$ for all $1<r<\infty,$ moreover, it can
	be proved that $M_{\phi} \simeq M^{l+1},$ where $\alpha=l \in \mathbb{N}$ and $M^{l+1}$ is $M$ iterated $l+1$ times.
	\item $M_{\phi}= M_{L(\log L)^{\alpha}(\log \log L)^{\beta}}$ given by the function $\phi(t)=t \log^{\alpha} (e+t) \log^{\beta} (e+\log (e+t))$ with $\alpha, \beta>0.$
\end{itemize}

\subsection{ Sparse family} We also need a system of dyadic calculus from \cite{ler2, ler3}, so in this subsection, we introduce a part of it.
\begin{definition}\label{def1.2}
	By a dyadic lattice $\mathcal{D},$ we mean a collection of cubes which satisfies the following properties:
	\begin{enumerate}[(i).]
		\item For any $Q \in \mathcal{D}$ its sidelength $\ell_{Q}$ is of the form $2^{k}, k \in \mathbb{Z}$;
		\item $Q \cap R \in\{Q, R, \emptyset\}$ for any $Q, R \in \mathcal{D}$;
		\item The cubes of a fixed sidelength $2^{k}$ form a partition of $\mathbb{R}^{n}$.
	\end{enumerate}

\end{definition}
An interesting and crucial theorem in dyadic calculus is Three Lattice Theorem, which asserts that given a dyadic lattice $\mathcal{D},$ there exist $3^{n}$ dyadic lattices $\mathcal{D}_{1}, \ldots, \mathcal{D}_{3^{n}}$ such that for each cube $Q \in \mathcal{D}$, we can find a cube $R_{Q}$ in some $\mathcal{D}_{j}$ such that $Q \subseteq R_{Q}$ and $3 l_{Q}=l_{R_{Q}}$.\par

According to the method of taking dyadic lattice $\mathcal{D}$ in Definition \ref{def1.2}, we can give the definition of sparse family $\mathcal{S}$ as follows.

\begin{definition}
	Let $\mathcal{D}$ be a dyadic lattice. $\mathcal{S} \subset \mathcal{D}$ is called a $\eta$-sparse family with $\eta \in(0,1)$ if for every cube $Q \in \mathcal{S},$
	$$\left|\bigcup_{P \in \mathcal{S}, P \subsetneq Q} P\right|\leq (1-\eta) \left| Q \right|.$$
\end{definition}
There is another equivalent definitions of sparsity for a collection of sets. If we define
$$E(Q)=Q \backslash \bigcup_{P \in \mathcal{S}, P \subsetneq Q} P,$$
then it is easy to deduce that the sets $E(Q)$ are pairwise disjoint and $|E(Q)| \geq \eta |Q|$.\par
Let $\mathcal{D}$ be a dyadic lattice and $\mathcal{S}\subseteq \mathcal{D}$ be a $\eta$-sparse family, the sparse operator is defined by
$$\mathcal{A}_{r, \mathcal{S}} f(x)= \sum_{Q \in \mathcal{S}}\langle|f|\rangle_{r,Q}\chi_{Q}(x)=\sum_{Q \in \mathcal{S}}\left(\frac{1}{|Q|} \int_{Q}|f(y)|^{r} d y\right)^{\frac{1}{r}} \chi_{Q}(x),$$
where $r>0$ and $\langle|f|\rangle_{r,Q}^r=\frac{1}{|Q|} \int_{Q}|f(y)|^{r} d y$. Furthermore, associated with the constants $\beta \in[0, \infty)$ and $r_{1}, r_{2}, r \in[1, \infty)$, we can define the bilinear sparse operators $\mathcal{A}_{\mathcal{S} ; L(\log L)^{\beta}, L^{r}}$ and $\mathcal{A}_{\mathcal{S}; L^{r_{1}}, L^{r_{2}}}$ by
$$
\mathcal{A}_{\mathcal{S} ; L(\log L)^{\beta}, L^{r}}(f, g)=\sum_{Q \in \mathcal{S}}|Q|\|f\|_{L(\log L)^{\beta}, Q}\langle|g|\rangle_{r,Q},
$$
$$
\mathcal{A}_{\mathcal{S}; L^{r_{1}}, L^{r_{2}}}(f, g)=\sum_{Q \in \mathcal{S}}|Q|\langle|f|\rangle_{r_{1},Q }\langle|g|\rangle_{r_{2},Q}.
$$\par
According to the above definitions, we can give the condition that the operator $T$ satisfies the bilinear sparse domination. More specifically, for $\beta, q \in(0, \infty)$, we say that a sublinear operator $T$ acting on $\cup_{p \geq 1} L^{p}\left(\mathbb{R}^{n}\right) $ enjoys a $\left(L(\log L)^{\beta}, L^{q}\right)$-bilinear sparse domination with bound $A$, if for each bounded function $f$ with compact support, there exists a sparse family $\mathcal{S}$ of cubes, such that
$$
\left|\int_{\mathbb{R}^{n}} g(x) T f(x) dx\right| \leq A \mathcal{A}_{\mathcal{S}, L(\log L)^{\beta}, L^{q}}(f, g),
$$
holds for all bounded function $g$.

The following results of bilinear sparse domination are crucial in our analysis (\cite[Corollary 5.1]{hu2}).
\begin{lemma}[\cite{hu2}]\label{lem1.2}
Let $\Omega_{1}, \Omega_{2}$ be homogeneous of degree zero, have mean value zero and $\Omega_{1}, \Omega_{2} \in L^{\infty}\left(\mathbb{S}^{n-1}\right)$. Let $r \in(1,3 / 2]$. Then for each bounded function $f$ with compact support, there exists a $\frac{1}{2} \frac{1}{9^n}$-sparse family of cubes $\mathcal{S}=\{Q\}$, and functions $J_{1}$ and $J_{2}$, such that for each function $g$,
\begin{equation}\label{ie1.1}
\begin{aligned}
&\left|\int_{\mathbb{R}^{n}} J_{1}(x) g(x) \mathrm{d} x\right| \lesssim r^{\prime} \mathcal{A}_{\mathcal{S} ; L (\log L), L^{r}}(f, g), \\
&\left|\int_{\mathbb{R}^{n}} J_{2}(x) g(x) \mathrm{d} x\right| \lesssim r^{\prime 2} \mathcal{A}_{\mathcal{S} ; L^{1}, L^{r}}(f, g),
\end{aligned}
\end{equation}
and for a. e. $x \in \mathbb{R}^{n}$,
\begin{equation}\label{ie1.2}
T_{\Omega_{1}} T_{\Omega_{2}} f(x)=J_{1}(x)+J_{2}(x).
\end{equation}
\end{lemma}
\par
Here, we give a brief introduction to sparse domination.
The study of sparse domination is a very active research area in Harmonic analysis in recent years. It helps to simplify the proof of some well  known results and even can be used to obtain the dependence of the norm constant on the weight functions. For example, Lerner \cite{ler20} introduced a class of sparse operators and gave an alternative and simple proof of the $A_2$ conjecture. Later on, Lerner \cite{ler2}  obtained quantitative weighted bounds for Calder\'{o}n-Zygmund operator $T$ which satisfies a H\"{o}lder-Lipschitz condition. Since then, great attentions have been paid to the study of the sparse bounds. We refer the readers to \cite{cej, con, cul}  and the references therein for more informations.\par

\section{ Proofs of Theorems \ref{thm1.1} }\label {Sect 3}

This section is devoted to prove Theorem \ref{thm1.1} and Corollary \ref{cor1.1}. To demonstrate Theorem \ref{thm1.1}, we need the following lemma in \cite[Lemma 3.3]{and}.
\begin{lemma}[\cite{and}]\label{lem1.3}
Let $\mathbb{X}$ be an RIQBFS which is $p$-convex for some $0<p \leq 1$. If $1<p_{\mathbb{X}}<\infty$, then for all $w \in A_{p_{\mathbb{X}}},$ we have
$$
\|M\|_{\mathbb{X}(w) \mapsto \mathbb{X}(w)} \leq C[w]_{A_{p_{\mathbb{X}}}}^{1 / p_{\mathbb{X}}},
$$
where $C$ is an absolute constant only depending on $p_{\mathbb{X}}$ and $n$.
\end{lemma}

\begin{proof}[Proof of Theorem $\ref{thm1.1}$]
We define the weighted dyadic Hardy-Littlewood maximal operators $M_{w}^{\mathcal{D}}$ and $M_{w,r}^{\mathcal{D}}$ by
	$$M_{w}^{\mathcal{D}} f(x):=\sup _{x \in R, R \in \mathcal{D}} \frac{1}{w(R)} \int_{R}|f(y)| w(y) d y, \quad f \in L_{\operatorname{loc}}^{1}\left(\mathbb{R}^{n}\right),$$

$$M_{w,r}^{\mathcal{D}} f(x):=\sup _{x \in R, R \in \mathcal{D}}
\left(\frac{1}{w(R)} \int_{R}|f(y)|^r w(y) d y\right)^{\frac{1}{r}}, \quad f \in L_{\operatorname{loc}}^{r}\left(\mathbb{R}^{n}\right),
$$
	where $w \in A_{\infty}$ and $\mathcal{D}$ is the given dyadic grid.\par
First, we consider the case of $ 1<p_{\X}\leq q_{\X}< 2-\frac{1}{1+p_{\X}}=:q_0.$
For a fixed $w \in A_{p_{\X}}$, by (\ref{ie1.2}), we obtain

\begin{equation*}
\begin{aligned}
\left\| T_{\Omega_1}T_{\Omega_2}f\right\|_{\X(w)} &=\sup_{\|g\|_{\X'(w)}\leq 1}\left|\int_{\mathbb{R}^{n}} T_{\Omega_1}T_{\Omega_2} f(x) g(x) w(x) d x\right|\\
&\leq \sup_{\|g\|_{\X'(w)}\leq 1}\left|\int_{\mathbb{R}^{n}} J_1(x) g(x) w(x) d x\right|+\sup_{\|g\|_{\X'(w)}\leq 1}\left|\int_{\mathbb{R}^{n}} J_2(x) g(x) w(x) d x\right|\\
&=:I_1+I_2,\\
\end{aligned}
\end{equation*}
where $I_i=\sup\limits_{\|g\|_{\X'(w)}\leq 1}\left|\int_{\mathbb{R}^{n}} J_i(x) g(x) w(x) d x\right|, i=1,2.$\par
Consider the contribuion of $I_1$.  Take and fix any $g \in \mathbb{X}^{\prime}(w)$ with $\|g\|_{\mathbb{X}^{\prime}(w)} \leq 1$, by (\ref{ie1.1}),
\begin{equation*}
\left|\int_{\mathbb{R}^{n}} J_1(x) g(x) w(x) d x\right|\lesssim r' \sum_{Q \in \mathcal{S}} \left\|f\right\|_{L(\log{L}),Q}{\langle|gw|\rangle}_{r,Q}|Q|.
\end{equation*}
A direct computation gives that
\begin{equation*}
\begin{aligned}
{\langle|gw|\rangle}_{r,Q}&=\left(\frac{1}{|Q|}\int_Q |g(x)w(x)|^r dx\right)^{\frac{1}{r}}=\left(\frac{1}{|Q|}\int_Q |g(x)|^r |w(x)|^{\frac{1}{s}}|w(x)|^{r-\frac{1}{s} } dx\right)^{\frac{1}{r}} \\
&\leq \left(\frac{1}{|Q|}\int_Q |g(x)|^{rs} w(x) dx\right)^{\frac{1}{rs}} \left(\frac{1}{|Q|}\int_Q |w(x)| ^{(r-\frac{1}{s})s'} dx\right)^{\frac{1}{rs'}}.
\end{aligned}
\end{equation*}
Take appropriate $r$ and $s$ such that $(r-\frac{1}{s})s'<1+\frac{1}{2^{11+n}[w]_{A_\infty}}.$ For example, let us choose
$r=1+\frac{1}{2^{13+n}p_{\X}[w]_{A_\infty}}, s=1+\frac{1}{2p_{\X}},$ then
$$(r-\frac{1}{s})s'=1+\frac{1}{2^{13+n}p_{\X}[w]_{A_\infty}}
+\frac{1}{2^{12+n}[w]_{A_\infty}}<1+\frac{1}{2^{11+n}[w]_{A_\infty}}.$$
Thus, combining Lemma \ref{lem1.1}, we deduce that
\begin{equation*}
\begin{aligned}
\frac{1}{|Q|}\int_{Q}(w(x))^{(r-\frac{1}{s})s'}dx &\leq \left(\frac{2}{|Q|}\int_{Q}w(x) dx \right)^{1-\frac{1}{rs}}.
\end{aligned}
\end{equation*}
Now one can get
\begin{equation*}
\begin{aligned}
{\langle|gw|\rangle}_{r,Q}\cdot|Q|&\leq  \left(\frac{1}{|w(Q)|}\int_Q |g(x)|^{rs}w(x) dx\right)^{\frac{1}{rs}} w(Q)=: g_{Q,w}^{rs}\cdot w(Q).
\end{aligned}
\end{equation*}\par
Taking into account the generalized H\"{o}lder's inequality and recalling that $M^2$ is $M$ iterated 2 times, we have
\begin{equation*}
\begin{aligned}
\sum_{Q \in \mathcal{S}} \left\|f\right\|_{L(\log L),Q}{\langle|gw|\rangle}_{r,Q}|Q|&\leq \sum_{Q \in \mathcal{S}} \left\|f\right\|_{L(\log L),Q}w(Q)g_{Q,w}^{rs} \\
&\leq \sum_{B \in \mathcal{B}} \left\|f\right\|_{L(\log L),B}g_{B,w}^{2s} \sum_{\substack{R \in \S\\\pi(R)=B}}w(R) \\
&\leq C(n)[w]_{A_\infty}\sum_{B \in \mathcal{B}} \left\|f\right\|_{L(\log L),B}w(B)g_{B,w}^{2s} \\
&=C(n)[w]_{A_{\infty}}  \int_{\R^{n}} \sum_{B \in \mathcal{B}} \left\|f \right\|_{L(\log L),B}g_{B,w}^{2s}\chi_{B}(x)w(x)dx \\
&\leq C(n)[w]_{A_{\infty}}  \int_{\R^{n}} M_{L(\log L)}f(x)M_{w,2s}^{\mathcal{D}}g(x)w(x)dx\\
&\leq C(n)[w]_{A_{\infty}}   \int_{\R^{n}} M^2 f(x)M_{w,2s}^{\mathcal{D}}g(x)w(x)dx\\
&\leq C(n)[w]_{A_{\infty}}  \left\| M^2f\right\|_{\X(w)} \left\| M_{w,2s}^{\mathcal{D}}g\right\|_{\X'(w)},
\end{aligned}
\end{equation*}
where $\mathcal{B}$ is the family of the principal cubes in the usual sense and $\pi(R)$ is the minimal principal cube which contains $R$. That is
	$$\mathcal{B}=\cup_{k=0}^{\infty} \mathcal{B}_{k}$$
	with $\mathcal{B}_{0}:=\{$ maximal cubes in $\mathcal{S}\}$ and
	$$\mathcal{B}_{k+1}:=\underset{B\in \mathcal{B}_{k}}{\cup} \operatorname{ch}_{\mathcal{B}}(B), \quad \operatorname{ch}_{\mathcal{B}}(B)=\{R \subsetneq B \text { maximal s.t. } \tau(R)>2 \tau(B)\},$$
	where $\tau(R)=\|f\|_{{L(\log L)}, R}g_{R, w}^{2s}$. \par
Now we observe that, using Lemma \ref{lem1.3},
\begin{equation*}
\begin{aligned}
I &\leq C(n)[w]_{A_\infty}^2  [w]_{A_{p_\X}}^{\frac{2}{p_\X}} \left\|f\right\|_{\X(w)} \left\|\left( M_{w}^{\mathcal{D}}(|g|^{2s})\right)^{\frac{1}{2s}} \right\|_{\X'(w)} \\
&= C(n)[w]_{A_\infty}^2  [w]_{A_{p_\X}}^{\frac{2}{p_\X}} \left\|f\right\|_{\X(w)} \left\|\left( M_{w}^{\mathcal{D}}(|g|^{2s})\right) \right\|_{\X'^{\frac{1}{2s}}(w)}^{\frac{1}{2s}} \\
&\lesssim[w]_{A_\infty}^2 [w]_{A_{p_\X}}^{\frac{2}{p_\X}} \left\|f\right\|_{\X(w)} \left\||g|^{2s} \right\|_{\X'^{\frac{1}{2s}}(w)}^{\frac{1}{2s}} \\
&\leq [w]_{A_\infty}^2 [w]_{A_{p_\X}}^{\frac{2}{p_\X}} \left\|f\right\|_{\X(w)},  \\
\end{aligned}
\end{equation*}
 where in the second inequality of the countdown, we have used the fact that the boundedness of $M_w^{\mathcal{D}}$( see \cite[Theorem 3.2]{cur}). Indeed, using the condition $q_{\X}<\frac{2s}{2s-1}$, it’s easy to deduce that $$p_{\X'^{\frac{1}{2s}}}=\frac{p_{{\X}'}}{2s}=\frac{{(q_\X)}^{'}}{2s}>1.$$\par

Now we turn our attention to the estimate of $I_2.$ For this case, arguing as in the first case, we obtain
\begin{equation*}
\begin{aligned}
\sum_{Q \in \mathcal{S}}\langle|f|\rangle_Q \langle|gw|\rangle_{r,Q}|Q|
&\lesssim  \sum_{Q \in \mathcal{S}}\langle|f|\rangle_Q g_{Q,w}^{rs}w(Q)\\
&\leq  \sum_{Q \in \mathcal{S}} \frac{1}{w(Q)}\left( \int_Q(Mf(x))^{\frac{1}{2}}(M_{w,2s}^{\mathcal{D}}g(x))^{\frac{1}{2}}w(x)dx \right)^2.
\end{aligned}
\end{equation*}
That fact together with Lemma \ref{lem1.2} yields
\begin{equation}\label{ie1.4}
\left|\int_{\mathbb{R}^n} J_2(x)g(x)w(x)dx\right| \lesssim r'^2\sum_{Q \in \mathcal{S}} \frac{1}{w(Q)}\left( \int_Q(Mf(x))^{\frac{1}{2}}(M_{w,2s}^{\mathcal{D}}g(x))^{\frac{1}{2}}w(x)dx \right)^2.
\end{equation}
Now we note that $\mathcal{S}$ is a $\frac{1}{2\cdot 9^n}$ -sparse family (i.e., for any $Q\in \mathcal{S},$ there exists $E(Q)$ such that $|E(Q)|\geq \frac{1}{2\cdot 9^n}|Q|$ ). Hence, for each dyadic cube $Q \in \mathcal{S}$, we have
\begin{equation}\label{ie1.3}
\begin{aligned}
\sum_{R \subseteq Q}w(R)=\sum_{R \subseteq Q}\frac{w(R)}{|R|}|R|
&\leq 2\cdot9^n \sum_{R \subseteq Q}\frac{w(R)}{|R|}\cdot|E(R)| \\
&\leq 2\cdot9^n \sum_{R \subseteq Q}\int_{E(R)} M(w\chi_R)(x)dx \\
&\leq  2\cdot9^n\int_Q M(w\chi_Q)(x)dx \\
&\leq  2\cdot9^n[w]_{A_\infty} w(Q).
\end{aligned}
\end{equation}
Combining (\ref{ie1.3}) with (\ref{ie1.4}) and the Carleson embedding theorem, together with the generalized H\"{o}lder's inequality, one may obtain
\begin{equation*}
\begin{aligned}
I_2  &\lesssim [w]_{A_\infty}^2\sup_{\|g\|_{\X'(w)}\leq 1}\sum_{Q \in \mathcal{S}} \frac{1}{w(Q)}\left( \int_Q(Mf(x))^{\frac{1}{2}}(M_{w,2s}^{\mathcal{D}}g(x))^{\frac{1}{2}}w(x)dx \right)^2\\
& \lesssim[w]_{A_\infty}^3  \left\| Mf\right\|_{\X(w)} \sup_{\|g\|_{\X'(w)}\leq 1}\left\|M_{w,2s}^{\mathcal{D}}g\right\|_{\X'(w)}\\
& \leq [w]_{A_\infty}^3   [w]_{A_{p_\X}}^{\frac{1}{p_\X}} \left\| f\right\|_{\X(w)} \sup_{\|g\|_{\X'(w)}\leq 1} \left\||g|^{2s}\right\|_{\X'^{\frac{1}{2s}}(w)}^{\frac{1}{2s}}  \\
&\leq[w]_{A_\infty}^3  [w]_{A_{p_\X}}^{\frac{1}{p_\X}} \left\| f\right\|_{\X(w)}.
\end{aligned}
\end{equation*}\par
This inequality, combined with the estimate of $I_1$ yields
\begin{equation*}
\left\| T_{\Omega_1}T_{\Omega_2}f\right\|_{\X(w)} \lesssim [w]_{A_\infty}^2  [w]_{A_{p_\X}}^{\frac{1}{p_\X}} \left( [w]_{A_\infty}+[w]_{A_{p_\X}}^{\frac{1}{p_\X}} \right) \left\| f\right\|_{\X(w)}.
\end{equation*}

To end the proof we consider the case of $ r<p_\X \leq q_\X<\infty,$ where $ 1<r<\infty $ is any fixed constant. \par
If $w\in A_{p_\X /r},$
\begin{equation*}
\begin{aligned}
\left\| T_{\Omega_1}f\right\|_{\X(w)} &=\sup _{\|g\|_{\X^{\prime}(w)} \leq 1}\left|\int_{\R^n}T_{\Omega_1}f(x)g(x)w(x)dx \right|\\
&\lesssim \sup _{\|g\|_{\X^{\prime}(w)} \leq 1}\sum_{Q\in{\mathcal{S}}} \langle |f| \rangle_{r,Q}\langle |gw| \rangle_{1,Q}|Q|,
\end{aligned}
\end{equation*}
where the last step follows from \cite[Theorem A]{con} or \cite[Corollary 3.4]{ler1}.\par
By a direct computation and the generalized H\"{o}lder's inequality, it follows that
\begin{equation*}
\begin{aligned}
\sum_{Q\in{\mathcal{S}}} \langle |f| \rangle_{r,Q}\langle |gw| \rangle_{1,Q}|Q|&= \sum_{Q\in{\mathcal{S}}} \langle |f| \rangle_{r,Q}\frac{1}{w(Q)}\int_Q |g(x)|w(x)dx w(Q) \\
&\lesssim [w]_{A_\infty} \int_{\R^n} M_rf(x)M_w^{\mathcal{D}}g(x)w(x)dx\\
&\lesssim[w]_{A_\infty} \left\|M_rf\right\|_{\X(w)}\left\|M_w^{\mathcal{D}}g\right\|_{\X'(w)} \\
&\lesssim[w]_{A_\infty} \left\|M_rf\right\|_{\X(w)} \left\|g\right\|_{\X'(w)},
\end{aligned}
\end{equation*}
where in the first inequality we apply the Carleson embedding theorem.\\
Hence, by the fact that $\left\|M_rf\right\|_{\X(w)} \lesssim [w]_{A_{p_{\X}/r}}^{\frac{1}{rq}} \left\| f \right\|_{\X(w)}$ (see \cite[Lemma 3.1]{tan}), where $q=p_{\X}/{r}-\varepsilon (p_{\X}/{r})$ and $\varepsilon (p)$ is defined in Lemma \ref{lem1.4}, we obtain
$$\left\|T_{\Omega_1}f\right\|_{\X(w)} \leq  C[w]_{A_\infty}  [w]_{A_{p_{\X}/{r}}}^{\frac{1}{rq}} \left\| f \right\|_{\X(w)}. $$    \\
This inequality, combining with the definition of $T_{\Omega_1}T_{\Omega_2}$ gives

 $$\left\|T_{\Omega_1}T_{\Omega_2}f\right\|_{\X(w)}\leq C[w]_{A_\infty}^2  [w]_{A_{p_{\X}/{r}}}^{\frac{2}{rq}} \left\| f \right\|_{\X(w)}, $$
This finish the proof of the second case.\par
Therefore, we obtain

\begin{equation*}\left\|T_{\Omega_1}T_{\Omega_2}f\right\|_{\X(w)}\lesssim\left\{\begin{array}{ll}
[w]_{A_\infty}^2  [w]_{A_{p_{\X}/{r}}}^{\frac{2}{rq}} \left\| f \right\|_{\X(w)}, &\text{ if } r<p_\X\leq q_\X, w\in A_{\frac{p_\X}{r}}; \\
{[w]_{A_\infty}^2}[w]_{A_{p_\X}}^{\frac{1}{p_\X}}\left([w]_{A_\infty}+[w]_{A_{p_\X}}^{\frac{1}{p_\X}}\right)
\left\|f\right\|_{\X(w)},& \text{ if } 1<p_\X<q_0, w\in A_{p_\X} ,
\end{array}\right.
\end{equation*}
where $q_0=2-\frac{1}{1+p_{\X}}$.
\end{proof}

Using Theorem \ref{thm1.1}, we can easily get Corollary $\ref{cor1.1}$.

\begin{proof}[Proof of Corollary $\ref{cor1.1}$]
Recall that if $\X$ be a RIQBFS with $p$-convex then $\X ^{\frac{1}{p}}$ is a RIBFS. This together with Theorem \ref{thm1.1}
lead to the desired result. As a matter of fact, it suffices for us to prove that for any $r>0$, $\left\||f|^r\right\|^{1/r}_{\X(w)}=\left\|f\right\|_{\X ^r(w)}$ and $p_{\X ^r}=r\cdot p_{\X}, q_{\X ^r}=r\cdot q_{\X}$ hold for RIQBFS $\X$. To see this, using the fact that $\X ^{1/p}$ is a RIBFS, we have
$$p_{\X ^r}=p_{\X ^{\frac{1}{p}pr}}=pr\cdot p_{\X ^\frac{1}{p}}. $$
On the other hand,
$$p_{\X}=p_{\X ^{\frac{1}{p}p}}=p\cdot p_{\X ^\frac{1}{p}}. $$
Therefore $p_{\X ^r}=r\cdot p_{\X}$ and it also works for $q_\X$.\\
From the definitions of the norm of $\X ^r,$ it is easy to check that
$$\left\|f\right\|_{\X ^r(w)}=\left\||f|^{pr}\right\|^{\frac{1}{pr}}_{\X ^{\frac{1}{p}}(w)}$$
and
$$\left\||f|^r\right\|_{\X (w)}=\left\||f|^{pr}\right\|^{\frac{1}{p}}_{\X ^{\frac{1}{p}}(w)}.$$
Then we have $\left\||f|^r\right\|^{1/r}_{\X(w)}=\left\|f\right\|_{\X ^r(w)}.$

\end{proof}

\section{ Proofs of Theorem \ref{thm1.2} }\label {Sect 4}
The section will be devoted to prove Theorem \ref{thm1.2}. Using sparse domination method, we will show that the bilinear sparse domination of the composition of rough singular integral operators holds, which implies Theorem \ref{thm1.2}.

We start with some definitions. For a linear operator $T$ and $1\leq r <\infty$, we define the corresponding grand maximal
truncated operator $\mathscr{M}_{T, r}$ by
$$
\mathscr{M}_{T, r} f(x):=\sup _{Q \ni x}|Q|^{-\frac{1}{r}}\left\|T\left(f \chi_{\mathbb{R}^{n} \backslash 3 Q}\right) \chi_{Q}\right\|_{L^{r}\left(\mathbb{R}^{n}\right)},
$$
where the supremum is taken over all cubes $Q \subset \mathbb{R}^{n}$ containing $x $. It is well known that the operator $\mathscr{M}_{T, r}$ was introduced by Lerner \cite{ler1} and it plays a key role in establishing bilinear sparse domination of rough operator. Let $T_{1}, T_{2}$ be two linear operators. We define the grand maximal operator $\mathscr{M}_{T_{1} T_{2}, r}^{*}$ by
$$
\mathscr{M}_{T_{1} T_{2}, r}^{*} f(x):=\sup _{Q \ni x}\left(\frac{1}{|Q|} \int_{Q}\left|T_{1}\left(\chi_{\mathbb{R}^{n} \backslash 3 Q} T_{2}\left(f \chi_{\mathbb{R}^{n} \backslash 9 Q}\right)\right)(\xi)\right|^{r} \mathrm{~d} \xi\right)^{\frac{1}{r}} .
$$

Now, we can state bilinear sparse domination of the composition operator associated with RIQBFS $\X .$
\begin{lemma}\label{lem2.1}
 Let $1<r<\frac{3}{2}$, $0\leq \beta_{1}, \beta_{2} < \infty.$ Let $T_{1}, T_{2}$ be two linear operators satisfying additionally the following conditions
 \begin{enumerate}[(i).]
		\item the operator $T_{1}$ is bounded on $L^{r^{\prime}}\left(\mathbb{R}^{n}\right)$ with bound $A$;
		\item there exists $A_0>0$ such that for each $\lambda>0$,
$$
\left|\left\{x \in \mathbb{R}^{n}:\left|T_{1} T_{2} f(x)\right|>\lambda\right\}\right| \lesssim \int_{\mathbb{R}^{n}} \frac{A_{0}|f(x)|}{\lambda} \log ^{\beta_{1}}\left(\mathrm{e}+\frac{A_{0}|f(x)|}{\lambda}\right) \mathrm{d} x;
$$

		\item there exists $A_1,A_2>0$ such that for each $\lambda>0$,
$$
\left|\left\{x \in \mathbb{R}^{n}: \mathscr{M}_{T_{1}, r^{\prime}} T_{2} f(x)>\lambda\right\}\right| \lesssim \int_{\mathbb{R}^{n}} \frac{A_{1}|f(x)|}{\lambda} \log ^{\beta_{1}}\left(\mathrm{e}+\frac{A_{1}|f(x)|}{\lambda}\right) \mathrm{d} x
$$
and
$$
\left|\left\{x \in \mathbb{R}^{n}: \mathscr{M}_{T_{2}, r^{\prime}} f(x)>\lambda\right\}\right| \lesssim \int_{\mathbb{R}^{n}} \frac{A_{2}|f(x)|}{\lambda} \log ^{\beta_{2}}\left(\mathrm{e}+\frac{A_{2}|f(x)|}{\lambda}\right) \mathrm{d} x .
$$
	\end{enumerate}
Then for each $0<p\leq 1$ and a bounded function $f$ with compact support, there exists a $ \frac{1}{2\cdot9^{n}}$-sparse family of cubes $\mathcal{S}=\{Q\}$, and functions $H_{1}$ and $H_{2}$, such that for each function $g$,
$$\int_{\mathbb{R}^{n}} |H_{1}(x)|^p |g(x)| d x \lesssim\left(A_{0}^p+A_{1}^p\right) \mathcal{A}^{(p,1)}_{\mathcal{S} ; L(\log L)^{\beta_{1}}, L^{(r^{'}/p)^{'}}}(f, g),$$
$$\int_{\mathbb{R}^{n}} |H_{2}(x)|^p |g(x)| d x \lesssim A^p A_{2}^p \mathcal{A}^{(p,1)}_{\mathcal{S} ; L(\log L)^{\beta_{2}}, L^{(r^{'}/p)^{'}}}(f, g),$$
and for a.e. $x \in \mathbb{R}^{n}$,
$$
T_{1} T_{2} f(x)=H_{1}(x)+H_{2}(x),
$$
where $$\mathcal{A}^{(a,b)}_{\mathcal{S} ; L(\log L)^{\beta}, L^{t}}(f, g):=\sum_{Q \in \mathcal{S}}|Q|\|f\|^a_{L(\log L)^{\beta}, Q}\langle|g|\rangle^b_{t, Q}$$
with $0\leq a,b, \beta <\infty, 1\leq t <\infty.$

\end{lemma}

\begin{proof}[Proof of Lemma $\ref{lem2.1}$]
  For any $0<p\leq 1$, in order to establish bilinear sparse domination over operator $(T_{1} T_{2} f)^p$, we are going to follow the scheme of the proof of \cite[Theorem 3.1]{ler1}, together with some ideas in \cite{hu1}. For a fixed cube
  $Q_{0}$, we should also consider a local version of operators $\mathscr{M}_{T_{2}, r^{\prime}}$ and $\mathscr{M}_{T_{1} T_{2}, r^{\prime}}^{*}$ by
$$
\mathscr{M}_{T_{2} ; r^{\prime} ; Q_{0}} f(x)=\sup _{Q \ni x, Q \subset Q_{0}}|Q|^{-\frac{1}{r^{\prime}}}\left\|\chi_{ Q} T_{2}\left(f \chi_{3Q_{0} \backslash 3 Q}\right)\right\|_{L^{r^{\prime}}\left(\mathbb{R}^{n}\right)}
$$
and
$$
\mathscr{M}_{T_{1} T_{2}, r^{\prime} ; Q_{0}}^{*} f(x)=\sup _{Q \ni x, Q \subset Q_{0}}\left(\frac{1}{|Q|} \int_{Q}\left|T_{1}\left(\chi_{\mathbb{R}^{n} \backslash 3 Q} T_{2}\left(f \chi_{9 Q_{0} \backslash 9 Q}\right)\right)(\xi)\right|^{r^{\prime}} \mathrm{d} \xi\right)^{\frac{1}{r^{\prime}}},
$$
respectively.
Then we define three sets $E_i$ with $ i=1,2,3$ by

$$E_{1}=\left\{x \in Q_{0}:\left|T_{1} T_{2}\left(f \chi_{9Q_{0}}\right)(x)\right|>D A_{0}\|f\|_{L(\log L)^{\beta_{1}}, 9 Q_{0}}\right\};$$
$$E_{2}=\left\{x \in Q_{0}: \mathscr{M}_{T_{2}, r^{\prime} ; Q_{0}} f(x)>D A_{2}\|f\|_{L(\log L)^{\beta_{2},}, 9 Q_{0}}\right\};$$
$$E_{3}=\left\{x \in Q_{0}: \mathscr{M}_{T_{1} T_{2}, r^{\prime} ; Q_{0}}^{*} f(x)>D A_{1}\|f\|_{L(\log L)^{\beta_{1}} ,9Q_{0}}\right\},$$
with $D$ a positive constant.
Let $E=\cup_{i=1}^{3} E_{i}.$
Hence, by our hypothesis $(ii),$  $(iii)$ and \cite[Lemma 4.3]{hu2},  taking $D$ large enough, we deduce that
$$
|E| \leq \frac{1}{2^{n+2}}\left|Q_{0}\right|.
$$
Now, by using the Calder\'{o}n-Zygmund decomposition to the function $\chi_{E}$ on $Q_{0}$ at height $\lambda=$ $\frac{1}{2^{n+1}}$, we obtain pairwise disjoint cubes $\left\{P_{l}\right\}_l \subset \mathcal{D}\left(Q_{0}\right)$ such that
$$
\chi_{E}(x) \leq \frac{1}{2^{n+1}}
$$
for a.e. $x \notin \bigcup_l P_{l}$. Together with this we immediately obtain $\left|E \backslash \bigcup_{l} P_{l}\right|=0.$
At the same time, for each $l\geq 1$, we also have
$$
\frac{1}{2^{n+1}}\left|P_{l}\right| \leq\left|P_{l} \cap E\right| \leq \frac{1}{2}\left|P_{l}\right|.
$$
Observe that $\sum_{l}\left|P_{l}\right| \leq 2^{n+1}|E|\leq\frac{1}{2}\left|Q_{0}\right|,$ $P_{l} \cap E^{c} \neq \emptyset$.
Let
\begin{equation*}
\begin{aligned}
G_{1}(x):=&T_{1} T_{2}\left(f \chi_{9 Q_{0}}\right)(x) \chi_{Q_{0} \backslash \cup_{l} P_{l}}(x)\\
&+\sum_{l} T_{1}\left(\chi_{\mathbb{R}^{n} \backslash 3 P_{l}} T_{2}\left(f \chi_{9 Q_{0} \backslash 9 P_{l}}\right)\right)(x) \chi_{P_{l}}(x).
\end{aligned}
\end{equation*}
Then we have the following claim:
$$
\int_{\mathbb{R}^{n}} |G_{1}(x)|^p |g(x)| dx \lesssim\left(A_{0}^p+A_{1}^p\right)\|f\|^p_{L(\log L)^{\beta_{1}}, 9 Q_{0}} \langle|g|\rangle_{(\frac{r'}{p})', Q_{0}}\left|Q_{0}\right|.
$$
Indeed, noting that $0< p\leq 1$ and the definition of $G_1$ we have
\begin{equation*}
\begin{aligned}
|G_{1}(x)|^p\leq&\left|T_{1} T_{2}\left(f \chi_{9 Q_{0}}\right)(x)\right|^p \chi_{Q_{0} \backslash \cup_{l} P_{l}}(x)\\
&+\sum_{l} |T_{1}\left(\chi_{\mathbb{R}^{n} \backslash 3 P_{l}} T_{2}\left(f \chi_{9 Q_{0} \backslash 9 P_{l}}\right)\right)(x)|^p \chi_{P_{l}}(x)\\
&=:L_1(x)+\sum_lL^p_{2,l}(x),
\end{aligned}
\end{equation*}
where $L_{2,l}(x)=|T_{1}\left(\chi_{\mathbb{R}^{n} \backslash 3 P_{l}} T_{2}\left(f \chi_{9 Q_{0} \backslash 9 P_{l}}\right)\right)(x)| \chi_{P_{l}}(x)$ with $l=1,2,\ldots$.\\
Therefore,
\begin{equation*}
\begin{aligned}
\int_{\mathbb{R}^{n}} |G_{1}(x)|^p |g(x)| dx&\leq \int_{Q_{0} \backslash \cup_{l} P_{l}} L_1(x) |g(x)| dx+
\sum_l\int_{P_l}L_{2,l}^p(x)|g(x)|dx\\
&\leq\left(\int_{Q_{0} \backslash E}+\int_{E \backslash \cup_{l} P_{l}}\right)L_1(x) |g(x)| dx+
\sum_l\int_{P_l}L_{2,l}^p(x)|g(x)|dx\\
&=\int_{Q_{0} \backslash E}L_1(x) |g(x)| dx+\sum_l\int_{P_l}L_{2,l}^p(x)|g(x)|dx,
\end{aligned}
\end{equation*}
where the last inequality is due to the fact that $\left|E \backslash \bigcup_{l} P_{l}\right|=0.$\par
Note that for any $x\in Q_0 \backslash E $ yields that $x \notin E_1$ which implies that
$$\left|T_{1} T_{2}\left(f \chi_{9Q_{0}}\right)(x)\right|\leq D A_{0}\|f\|_{L(\log L)^{\beta_{1}}, 9 Q_{0}}.$$
This estimate combined with H\"{o}lder's inequality for any $t\geq 1$ yields
\begin{equation}\label{ie1.5}
\begin{aligned}
\int_{Q_{0} \backslash E}L_1(x) |g(x)| dx&\leq D^p A_{0}^p\|f\|^p_{L(\log L)^{\beta_{1}}, 9 Q_{0}}{\langle|g|\rangle}_{1,Q_0}|Q_0|\\
&\leq D^p A_{0}^p\|f\|^p_{L(\log L)^{\beta_{1}}, 9 Q_{0}}{\langle|g|\rangle}_{t,Q_0}|Q_0|.
\end{aligned}
\end{equation}
For each $l$, using the fact that $P_{l} \cap E^{c} \neq \emptyset,$ we observe that there exists some $x_l\in P_{l} \cap E^{c}$ such that
$$
\inf _{\xi \in P_{l}}\mathscr{M}_{T_{1} T_{2}, r^{\prime} ; Q_{0}}^{*} f(\xi)\leq \mathscr{M}_{T_{1} T_{2}, r^{\prime} ; Q_{0}}^{*} f(x_l).
$$
Combining this inequality and by the H\"{o}lder's inequality with $\frac{r'}{p}\geq r'>1$, it gives
\begin{equation}\label{ie1.6}
\begin{aligned}
\sum_l\int_{P_l}L_{2,l}^p(x)|g(x)|dx&\lesssim \sum_l\left(\int_{P_l}L_{2,l}^{r'}(x)dx\right)^{\frac{p}{r'}}\left(\int_{P_l}|g(x)|^{(\frac{r'}{p})'}dx\right)
^{\frac{1}{(\frac{r'}{p})'}}\\
&=\sum_l\left(\frac{1}{|P_l|}\int_{P_l}L_{2,l}^{r'}(x)dx\right)^{\frac{p}{r'}}|P_l|^{\frac{p}{r'}}
\left(\int_{P_l}|g(x)|^{(\frac{r'}{p})'}dx\right)^{\frac{1}{(\frac{r'}{p})'}}\\
&\leq \sum_l\left(\inf _{\xi \in P_{l}}\mathscr{M}_{T_{1} T_{2}, r^{\prime} ; Q_{0}}^{*} f(\xi)\right)^p
|P_l|^{\frac{p}{r'}}\left(\int_{P_l}|g(x)|^{(\frac{r'}{p})'}dx\right)^{\frac{1}{(\frac{r'}{p})'}}\\
&\leq D^p A_{1}^p\|f\|^p_{L(\log L)^{\beta_{1}}, 9 Q_{0}}\sum_l
|P_l|^{\frac{p}{r'}}\left(\int_{P_l}|g(x)|^{(\frac{r'}{p})'}dx\right)^{\frac{1}{(\frac{r'}{p})'}}\\
&\leq D^p A_{1}^p\|f\|^p_{L(\log L)^{\beta_{1}}, 9 Q_{0}}\left(\sum_l|P_l|\right)^{\frac{p}{r'}}
\left(\sum_l\int_{P_l}|g(x)|^{(\frac{r'}{p})'}dx\right)^{\frac{1}{(\frac{r'}{p})'}}\\
&\leq A_{1}^p\|f\|^p_{L(\log L)^{\beta_{1}}, 9 Q_{0}}\left|Q_0\right|^{1-\frac{1}{(\frac{r'}{p})'}}
\left(\int_{Q_0}|g(x)|^{(\frac{r'}{p})'}dx\right)^{\frac{1}{(\frac{r'}{p})'}}\\
&=A_{1}^p\|f\|^p_{L(\log L)^{\beta_{1}}, 9 Q_{0}}\langle|g|\rangle_{(\frac{r'}{p})', Q_{0}}\left|Q_{0}\right|.
\end{aligned}
\end{equation}
Combining this bounds with  (\ref{ie1.5}), we see that our claim holds with $t=(\frac{r'}{p})'$.\par

Let $G_2(x)$ be the function defined by
$$
G_{2}(x):=\sum_{l} T_{1}\left(\chi_{3 P_{l}} T_{2}\left(f \chi_{9Q_{0} \backslash 9 P_{l}}\right)\right)(x) \chi_{P_{l}}(x).
$$
For each function $g$, applying the H\"{o}lder's inequality and the $L^{r'}$ boundedness of $T_1$ yield
\begin{equation*}
\begin{aligned}
\int_{\mathbb{R}^{n}} |G_{2}(x)|^p |g(x)| dx
&\leq \sum_l\left(\int_{P_l}\left|T_{1}\left(\chi_{3 P_{l}} T_{2}\left(f \chi_{9 Q_{0} \backslash 9 P_{l}}\right)\right)(x)\right|^{r'}(x)dx\right)^{\frac{p}{r'}}\langle|g|\rangle_{(\frac{r'}{p})', P_{l}}\left|P_{l}\right|^{\frac{1}{(\frac{r'}{p})'}}\\
&\leq A^p \sum_{l}\left(\int_{3 P_{l}}\left|T_{2}\left(f \chi_{9 Q_{0} \backslash 9 P_{l}}\right)(x)\right|^{r^{\prime}} dx\right)^{\frac{p}{r^{\prime}}}\langle|g|\rangle_{(\frac{r'}{p})', P_{l}}\left|P_{l}\right|^{\frac{1}{(\frac{r'}{p})'}}\\
&\leq 3^n A^p \sum_{l}\left(\frac{1}{|3P_l|}\int_{3 P_{l}}\left|T_{2}\left(f \chi_{9 Q_{0} \backslash 9 P_{l}}\right)(x)\right|^{r^{\prime}} dx\right)^{\frac{p}{r^{\prime}}}\langle|g|\rangle_{(\frac{r'}{p})', P_{l}}\left|P_{l}\right|.\\
\end{aligned}
\end{equation*}
The same reasoning as what we have done for $G_1$ then gives
\begin{equation}\label{ie1.7}
\begin{aligned}
\int_{\mathbb{R}^{n}} |G_{2}(x)|^p |g(x)| dx
&\lesssim A^p\sum_l\left(\inf _{\xi \in P_{l}}\mathscr{M}_{T_{2}, r^{\prime} ; Q_{0}}^{*} f(\xi)\right)^p
|P_l|^{\frac{p}{r'}}\left(\int_{P_l}|g(x)|^{(\frac{r'}{p})'}dx\right)^{\frac{1}{(\frac{r'}{p})'}}\\
&\lesssim A^p A_{2}^p\|f\|^p_{L(\log L)^{\beta_{2}}, 9 Q_{0}}\sum_l
|P_l|^{\frac{p}{r'}}\left(\int_{P_l}|g(x)|^{(\frac{r'}{p})'}dx\right)^{\frac{1}{(\frac{r'}{p})'}}\\
&\leq A^p A_{2}^p\|f\|^p_{L(\log L)^{\beta_{2}}, 9 Q_{0}}\left(\sum_l|P_l|\right)^{\frac{p}{r'}}
\left(\sum_l\int_{P_l}|g(x)|^{(\frac{r'}{p})'}dx\right)^{\frac{1}{(\frac{r'}{p})'}}\\
&\leq A^pA_{2}^p\|f\|^p_{L(\log L)^{\beta_{2}}, 9 Q_{0}}\langle|g|\rangle_{(\frac{r'}{p})', Q_{0}}\left|Q_{0}\right|.
\end{aligned}
\end{equation}\par
We note that at each point $x\in \mathbb{R}^n,$
$$
T_{1} T_{2}\left(f \chi_{9 Q_{0}}\right)(x) \chi_{Q_{0}}(x)=G_{1}(x)+G_{2}(x)+\sum_{l} T_{1} T_{2}\left(f\chi_{9 P_{l}}\right)(x) \chi_{P_{l}}(x).
$$
Observe that the last term on the right-hand side is consistent with the form on the left-hand side, so we can iterate with $T_{1} T_{2}\left(f\chi_{9 P_{l}}\right)(x) \chi_{P_{l}}(x)$ instead of $T_{1} T_{2}\left(f\chi_{9 Q_{0}}\right)(x) \chi_{Q_{0}}(x),$ and so on. For fixed $j_{1}, \ldots, j_{m-1}\in \mathbb{Z}^+$, let $\{Q_{0}^{j_{1} \ldots j_{m-1} j_{m}}\}_{j_{m}}$ be the cubes obtained at the $m$-th stage of the decomposition process to the cube $Q_{0}^{j_{1} \ldots j_{m-1}}$, where $\{Q_{0}^{j_{1}}\}=\{P_{j}\}$ . For each fixed $j_{1} \ldots, j_{m}$, define the functions $G_{Q_{0}, 1}^{j_{1} \ldots j_{m}} f$ and $G_{Q_{0}, 2}^{j_{1} \ldots j_{m}} f$ by
$$
G_{Q_{0}, 1}^{j_{1} \ldots j_{m}} f(x)=T_{1}\left(\chi_{\mathbb{R}^{n} \backslash 3 Q_{0}^{j_{1} \ldots j_{m}}} T_{2}\left(f \chi_{9 Q_{0}^{j_{1} \ldots j_{m-1}} \backslash 9 Q_{0}^{j_{1} \ldots j_{m}}}\right)\right)(x) \chi_{Q_{0}^{j_{1} \ldots j_{m}}}(x)
$$
and
$$
G_{Q_{0}, 2}^{j_{1} \ldots j_{m}} f(x)=T_{1}\left(\chi_{3 Q_{0}^{j_{1} \ldots j_{m}}} T_{2}\left(f \chi_{9 Q_{0}^{j_{1} \ldots j_{m-1}} \backslash 9 Q_{0}^{j_{1} \ldots j_{m}}}\right)\right)(x) \chi_{Q_{0}^{j_{1} \ldots j_{m}}}(x),
$$
respectively. \\
Let $\mathcal{F}=\left\{Q_{0}\right\} \cup_{m=1}^{\infty} \cup_{j_{1}, \ldots, j_{m}}\left\{Q_{0}^{j_{1} \ldots j_{m}}\right\}$. It is easy to check that $\mathcal{F} \subset \mathcal{D}\left(Q_{0}\right)$ is a $\frac{1}{2}$-sparse family with $\sum_{l}\left|P_{l}\right| \leq \frac{1}{2}\left|Q_{0}\right|$. Then
\begin{equation}
\begin{aligned}
G_{Q_{0}, 1}(x)=T_{1} &T_{2}\left(f \chi_{9 Q_{0}}\right) \chi_{Q_{0} \backslash \cup_{j_{1}} Q_{0}^{j_{1}}}(x)\\
&+\sum_{m=1}^{\infty} \sum_{j_{1}, \ldots, j_{m}} T_{1} T_{2}\left(f \chi_{9 Q_{0}^{j_{1} \ldots j_{m}}}\right) \chi_{Q_{0}^{j_{1} \ldots j_{m}} \backslash \cup_{j_{m+1}} Q_{0}^{j_{1} \ldots j_{m+1}}}(x)\\
&+\sum_{m=1}^{\infty} \sum_{j_{1}, \ldots, j_{m}} G_{Q_{0}, 1}^{j_{1} \ldots j_{m}} f(x) \chi_{Q_{0}^{j_{1} \ldots j_{m}}}(x) .
\end{aligned}
\end{equation}
Similar as what we have done for $G_2$, we may define function $G_{Q_{0}, 2}$ by
$$
G_{Q_{0}, 2}(x)=\sum_{m=1}^{\infty} \sum_{j_{1} \ldots j_{m}} G_{Q_{0}, 2}^{j_{1} \ldots j_{m}} f(x) \chi_{Q_{0}^{j_{1} \ldots j_{m}}}(x).
$$
Then for a. e. $x \in \mathbb{R}^n$,
$$
T_{1} T_{2}\left(f \chi_{9 Q_{0}}\right)(x)\chi_{Q_{0}}(x)=G_{Q_{0}, 1}(x)+G_{Q_{0}, 2}(x) .
$$\par
We are now ready to combine all our ingredients to finish the proof. In fact, by applying (\ref{ie1.6}), (\ref{ie1.7}) and (\ref{ie1.5}) with $t=(\frac{r'}{p})',$  we obtain
\begin{equation}\label{ie1.8}
  \int_{\mathbb{R}^{n}} |G_{Q_{0}, 1}(x)|^p |g(x)| dx \lesssim\left(A_{0}^p+A_{1}^p\right)\sum_{Q\in \mathcal{F}}\|f\|^p_{L(\log L)^{\beta_{1}}, 9 Q} \langle|g|\rangle_{(\frac{r'}{p})', Q}\left|Q\right|,
\end{equation}
\begin{equation}\label{ie1.9}
  \int_{\mathbb{R}^{n}} |G_{Q_{0}, 2}(x)|^p |g(x)| dx \lesssim\left(A^pA_{2}^p\right)\sum_{Q\in \mathcal{F}}\|f\|^p_{L(\log L)^{\beta_{2}}, 9 Q} \langle|g|\rangle_{(\frac{r'}{p})', Q}\left|Q\right|.
\end{equation}
Observe that $\bigcup_{l} Q_{l}=\mathbb{R}^{n}$, where the cubes $Q_l$'s have disjoint interiors and $\supp f \subset 9 Q_{j}$ for each $l$. To see this, we begin by taking a cube $Q_{0}$ such that $\supp f \subset Q_{0}$. And cover $9 Q_{0} \backslash Q_{0}$ by $9^{n}-1$ congruent cubes $Q_{l}$. For every $l, Q_{0} \subset 9 Q_{l}$. We continue to do the same way for $27 Q_{0} \backslash 9 Q_{0}$ and so on. It's easy to check that the union of the cubes $Q_{l}$ of this process, including $Q_{0}$, satisfies our requirement. Applying to each $Q_{l}$, we have that the above estimates (\ref{ie1.8}) and (\ref{ie1.9}) hold for $\mathcal{F}_l.$ Moreover, for a.e. $x \in \mathbb{R}^n$,
\begin{equation}\label{ie1.10}
T_{1} T_{2} f(x)=\sum_{l} G_{Q_{l}, 1} f(x)+\sum_{l} G_{Q_{l}, 2} f(x)=:H_{1} f(x)+H_{2} f(x).
\end{equation}
Let $\mathcal{S}$ denote $\left\{9Q: Q \in \bigcup_{l}\mathcal{F}_l\right\}.$ From the definitions of $\mathcal{F}_l$, it is easy to check that $\mathcal{S}$ is a $\frac{1}{2\cdot9^{n}}$-sparse family. This fact, together with (\ref{ie1.10}) gives us the desired results.
\end{proof}

Applying Lemma $\ref{lem2.1}$ to the rough singular integral operators $T_{\Omega_{1}}$ and $T_{\Omega_{2}}$, we can complete the proof of Theorem \ref{thm1.2}.
\begin{proof}[Proof of Theorem $\ref{thm1.2}$]
Let $T_{1}, T_{2}$ be two linear operators satisfying the conditions in Lemma \ref{lem2.1}. Note that $\X$ is a RIQBFS with $p$-convex, which implies that $\mathbb{Y}=\X ^{\frac{1}{p}}$ is a RIBFS. Then Lemma \ref{lem2.1} tells us that
 for a.e.$ x\in \R^n,$ $ T_1T_2f(x)=H_1(x)+H_2(x).$ Therefore,
\begin{equation}\label{ie2.1}
\begin{aligned}
\left\| T_1T_2f\right\|_{\X(w)} &\lesssim \left\| H_1 \right\|_{\X(w)}+ \left\| H_2 \right\|_{\X(w)} \\
&\backsimeq \sum_{i=1}^{2}\sup_{\left\|g\right\|_{\Y'(w)}\leq 1} \left( \int_{\R^n}|H_i(x)|^pg(x)w(x)dx \right)^{\frac{1}{p}}\\
&=: \sup_{\left\|g\right\|_{\Y'(w)}\leq 1} \left(\mathcal{L}_1(g) \right)^{\frac{1}{p}} +\sup_{\left\|g\right\|_{\Y'(w)}\leq 1} \left(\mathcal{L}_2(g) \right)^{\frac{1}{p}},
\end{aligned}
\end{equation}
where $\mathcal{L}_i(g)=\int_{\R^n}|H_i(x)|^pg(x)w(x)dx,i=1,2.$\par
Consider first the estimate of $\mathcal{L}_1(g).$ By Lemma \ref{lem2.1}, we have
\begin{equation*}
\begin{aligned}
\mathcal{L}_1(g) &\lesssim (A_0^p+A_1^p) \mathcal{A}^{(p,1)}_{\mathcal{S} ; L(\log L)^{\beta_{1}}, L^{(r^{'}/p)^{'}}}(f, gw)\\
&=(A_0^p+A_1^p)\sum_{Q \in \mathcal{S}} \left\|f\right\|_{L(\log L)^{\beta_1},Q}^p \langle|gw|\rangle_{(\frac{r'}{p})',Q}|Q|.
\end{aligned}
\end{equation*}

Observe that $(\frac{r'}{p})'\leq r $ for $0<p\leq 1,$ and then a direct calculation shows that
\begin{equation*}
|Q| \langle|gw|\rangle_{(\frac{r'}{p})',Q} \leq  \langle|gw|\rangle_{r,Q} |Q| \leq w(Q)\cdot g_{Q,w}^{rs},
\end{equation*}
where the definitions of $s$ and $g_{Q,w}^{rs}$ are as the same as in the proof of Theorem \ref{thm1.1}.\\
Hence, by the Carleson embedding theorem,
\begin{equation*}
\begin{aligned}
\mathcal{L}_1(g) &\lesssim (A_0^p+A_1^p) \sum_{Q \in \mathcal{S}} \left\|f\right\|_{L(\log L)^{\beta_1},Q}^p w(Q)g_{Q,w}^{rs} \\
&\lesssim (A_0^p+A_1^p) [w]_{A_\infty}  \int_{\R^n} \left( M_{L(\log L)^{\beta_1}}f(x) \right)^p M_{w,2s}^{\mathcal{D}}g(x)w(x)dx \\
&\lesssim (A_0^p+A_1^p) [w]_{A_\infty}  \int_{\R^n} \left( M^{\beta_1+1}f(x) \right)^p M_{w,2s}^{\mathcal{D}}g(x)w(x)dx,
\end{aligned}
\end{equation*}
which, together with the generalized H\"{o}lder's inequality, implies that

\begin{equation*}
\begin{aligned}
\mathcal{L}_1(g) &\lesssim (A_0^p+A_1^p) [w]_{A_\infty} \left\|( M^{\beta_1+1}f)^p \right\|_{\Y(w)} \left\| M_{w,2s}^{\mathcal{D}}g\right\|_{\Y'(w)} \\
&=(A_0^p+A_1^p) [w]_{A_\infty} \left\| M^{\beta_1+1}f \right\|_{\X(w)}^p \left\| M_w^{\mathcal{D}}(g^{2s})\right\|_{\Y'^{\frac{1}{2s}}(w)}^{\frac{1}{2s}}\\
&\lesssim (A_0^p+A_1^p) [w]_{A_\infty}[w]_{A_{p_\X}}^{\frac{\beta_1+1}{p_\X}p}\left\|f \right\|^p_{\X(w)},
\end{aligned}
\end{equation*}
where the last inequality follows from Lemma \ref{lem1.3} and the fact $p_{\Y'^{\frac{1}{2s}}}=\frac{p_{\Y'}}{2s}=\frac{1}{2s}\frac{q_{\X}}{q_{\X}-p}>1.$\\
Therefore
\begin{equation}\label{ie2.2}
 \left\| H_1 \right\|_{\X(w)} \lesssim (A_0^p+A_1^p)^{\frac{1}{p}}[w]_{A_{\infty}}^{\frac{1}{p}}[w]_{A_{p_\X}}^{\frac{\beta_1+1}{p_\X}}\left\|f \right\|_{\X(w)}.
\end{equation}
Now we turn to the proof of $\left\| H_2 \right\|_{\X(w)}.$ The same reasoning as what we have done for $\mathcal{L}_1(g)$ yields that
$$\mathcal{L}_2(g)=\int_{\R^n} |H_2(x)|^p|g(x)|w(x)dx \lesssim A^pA_2^p[w]_{A_\infty} [w]_{A_{p_\X}}^{\frac{\beta_2+1}{p_\X}p}\left\|f \right\|_{\X(w)}^p\left\|g \right\|_{\Y'(w)}.$$
By taking the supermum over $\|g\|_{\Y'(w)}\leq 1,$ we obtain
\begin{equation}\label{ie2.3}
\left\| H_2 \right\|_{\X(w)} \lesssim A A_2 [w]_{A_\infty}^{\frac{1}{p}}  [w]_{A_{p_\X}}^{\frac{\beta_2+1}{p_\X}}\left\|f \right\|_{\X(w)}.
\end{equation}
Combining the above estimates (\ref{ie2.1}), (\ref{ie2.2}) and (\ref{ie2.3}), we conclude that
$$\left\| T_1T_2f \right\|_{\X(w)}\lesssim \left[(A_0^p+A_1^p)^{\frac{1}{p}}+AA_2\right][w]_{A_{\infty}}^{\frac{1}{p}}\left([w]_{A_{p_\X}}^{\frac{\beta_1+1}{p_\X}}
+[w]_{A_{p_\X}}^{\frac{\beta_2+1}{p_\X}}\right)\left\|f \right\|_{\X(w)}.$$
Finally, we set $T_1=T_{\Omega_1},T_2=T_{\Omega_2},A_0=1,A_1=A=A_2=r',\beta_1=1,\beta_2=0.$ This together with the proof of Corollary 5.1 in \cite{hu2} yields
$$\left\|T_{\Omega_1}T_{\Omega_2}f\right\|_{\X(w)}\lesssim \left([w]_{A_\infty}^{1+\frac{1}{p}}+[w]_{A_\infty}^{2+\frac{1}{p}}\right)  \left([w]_{A_{p_\X}}^{\frac{1}{p_\X}}+[w]_{A_{p_\X}}^{\frac{2}{p_\X }}\right)\left\|f \right\|_{\X(w)}. $$
This finishes the proof of Theorem \ref{thm1.2}.
\end{proof}
\section{ Applications}\label {Sect 5}
This section will be devoted to give an application of Theorem \ref{thm1.1}. We consider the boundedness of certain non-standard Calder\'{o}n-Zygmund operators with rough kernels in RIBFS $\X$. For fixed $n\geq 2$, let $\Omega$ be a function with homogeneous of degree zero, integrable on the unit sphere $\mathbb{S}^{n-1}$ and satisfy the vanishing moment condition that for all $1 \leq j \leq n$,
\begin{equation}\label{e.2}
\int_{\mathbb{S}^{n-1}} \Omega\left(x\right) x_{j} d\sigma(x)=0,
\end{equation}
where $x_{j}(1 \leq j \leq n)$ denote the $j$-th variable of $x\in \mathbb{R}^{n}$.
Note that the vanishing condition here is different from (\ref{e.1}). \par
Let $A$ be a function on $\mathbb{R}^{n}$ whose derivatives of order one in $\mathrm{BMO}\left(\mathbb{R}^{n}\right),$ namely, $\nabla A \in \mathrm{BMO}$. Then we can define the non-standard rough Calder\'{o}n-Zygmund operator $T_{\Omega, A }$ by
\begin{equation}
T_{\Omega, A} f(x)=\text {p.v.} \int_{\mathbb{R}^{n}} \frac{\Omega(x-y)}{|x-y|^{n+1}}\left(A(x)-A(y)-\nabla A(y)\cdot(x-y)\right) f(y) d y.
\end{equation}
The dual operator of $T_{\Omega, A}$ has the following form
$$
\widetilde{T}_{\Omega, A} f(x)=\text {p.v.} \int_{\mathbb{R}^{n}} \frac{\Omega(x-y)}{|x-y|^{n+1}}(A(x)-A(y)-\nabla A(x)\cdot(x-y)) f(y) d y.
$$
The operator ${T}_{\Omega, A}$ is closely related to the Calder\'{o}n commutator, of interest in PDE, and was first studied by Cohen \cite{coh}. An interesting aspect of this operator is that it may not satisfy the classical standard kernel condition, even if the kernel $\Omega$ is a smooth kernel. This is also the main reason why people call it the non-standard singular integral operator. We refer the reader to \cite{coh,hu4, hu5} and their references for more details on this topic.
It is worth mentioning that Hu et al. \cite{hu6} recently got the endpoint $L \log L$ type estimate and the $L^p$ boundedness of ${T}_{\Omega, A}$ with $\Omega \in L(\log L)^{2}(\mathbb {S}^{n-1}).$ In addition, they also obtained the following results:

\begin{thC}[\cite{hu6}] Let $\Omega \in L^{\infty}\left(\mathbb{S}^{n-1}\right)$ be homogeneous of degree zero, satisfy the vanishing condition (\ref{e.2}), and $A$ be a function on $\mathbb{R}^{n}$ with derivatives of order one in $\mathrm{BMO}\left(\mathbb{R}^{n}\right)$. Then for $p \in(1, \infty)$ and $w \in A_{p}\left(\mathbb{R}^{n}\right)$, the following weighted norm inequality holds
$$
\left\|T_{\Omega, A} f\right\|_{L^{p}\left(\mathbb{R}^{n}, w\right)} \lesssim[w]_{A_{p}}^{\frac{1}{p}}\left([w]_{A_{\infty}}^{\frac{1}{p}}+[\sigma]_{A_{\infty}}^{\frac{1}{p}}\right)
[\sigma]_{A_{\infty}} \min \left\{[\sigma]_{A_{\infty}},[w]_{A_{\infty}}\right\}\|f\|_{L^{p}\left(\mathbb{R}^{n}, w\right)}.
$$
\end{thC}

As in \cite[Theorem 4.11]{hu6} and \cite[Theorem 5.6]{hu6}, by Theorem \ref{thm1.1}, we obtain
\begin{theorem}\label{thm4.1}
Let $\Omega$ be homogeneous of degree zero, have the vanishing moment (\ref{e.2}) and $\Omega\in L^{\infty}\left(\mathbb{S}^{n-1}\right)$, and $A$ be a function on $\mathbb{R}^{n}$ with derivatives of order one in $\mathrm{BMO}\left(\mathbb{R}^{n}\right)$. Let $1<r<\infty$ and $\X$ be a RIBFS with $1 <p_{\X}\leq q_{\X}< \infty$, then there exist $q>1$ such that
\begin{equation*}\left\|T_{\Omega, A} f\right\|_{\X(w)}\lesssim\left\{\begin{array}{ll}
[w]_{A_\infty}  [w]_{A_{p_\X/r}}^{\frac{1}{rq}} \left\| f \right\|_{\X(w)}, &\text{ if } r<p_\X\leq q_\X<\infty, w\in A_{\frac{p_\X}{r}}; \\
{[w]_{A_\infty}^2}[w]_{A_{p_\X}}^{\frac{2}{p_\X}}
\left\|f\right\|_{\X(w)},& \text{ if } 1<p_\X\leq q_\X<2-\frac{1}{1+p_\X}, w\in A_{p_\X}.
\end{array}\right.
\end{equation*}
\end{theorem}

\end{document}